\documentclass[11pt, hyperref]{article}
\usepackage[a4paper]{geometry}
\usepackage{hyperref}
\usepackage{authblk}
\usepackage{graphicx}
\usepackage{subfigure}
\usepackage{epstopdf}
\usepackage{amsmath}
\usepackage{amsthm}
\usepackage{amsfonts}
\usepackage{amssymb}
\usepackage{bm}
\usepackage{pifont}
\usepackage{circuitikz}
\pgfrealjobname{main}
\usepackage{caption}

\DeclareCaptionType{process}[Method]

\newtheorem{definition}{Definition}[section]
\newtheorem{lemma}{Lemma}[section]

\newtheorem{theorem}{Theorem}[section]

\newtheorem{remark}{Remark}[section]
\numberwithin{equation}{section}
\allowdisplaybreaks[4]

\usepackage{ulem}

\newcommand{\diag}{\mathrm{diag}}
\newcommand{\al}{\alpha}
\newcommand{\ep}{\epsilon}
\newcommand{\m}{\mathbb}
\newcommand{\om}{\omega}
\newcommand{\Ga}{\Gamma}
\newcommand{\la}{\lambda}
\newcommand{\La}{\Lambda}
\newcommand{\vphi}{\varphi}
\newcommand{\wt}{\widetilde}
\newcommand{\rank}{\mathrm{rank}}

\title{Effective Numerical Simulation of Fault Transient System}
\author[1, 2, 3]{Sixu Wu}
\author[3]{Feng Ji}
\author[3]{Lu Gao}
\author[4]{Ruili Zhang}
\author[5]{Cunwei Tang}
\author[1, 2]{Yifa Tang\footnote{Corresponding Author}}
\affil[1]{LSEC, ICMSEC, Academy of Mathematics and Systems Science, 

Chinese Academy of Sciences, Beijing 100190, China}
\affil[2]{School of Mathematical Sciences, University of Chinese Academy of Sciences, 

Beijing 100049, China}
\affil[3]{State Key Laboratory of Advanced Power Transmission Technology, 

Beijing 102209, China}
\affil[4]{School of Mathematics and Statistics, Beijing Jiaotong University,
	
Beijing 100044, China}
\affil[5]{Urban Power Supply Branch, State Grid Beijing Electric Power Company, 

Bejing 100031, China}

\date{}

\begin{document}
\maketitle

\begin{abstract}
  Power systems, including synchronous generator systems, are typical systems that strive for stable operation. In this article, we numerically study the fault transient  process of a synchronous generator system based on the first benchmark model. That is, we make it clear whether an originally stable generator system can restore its stability after a short time of unstable transient process. To achieve this, we construct a structure-preserving method and compare it with the existing and frequently-used predictor-corrector method. We newly establish a reductive form of the circuit system and accelerate the reduction process. Also a switching method between two stages in the fault transient process is given. Numerical results show the effectiveness and reliability of our method. 
\end{abstract}
\section{Introduction}
Fault transient system is a typical circuit system in which people focus on its transient stability, that is to say whether a stable system can reach a new stable state through some unstable transient process. The transient stability depends on both the original state of the system and the interference way\cite{PSSC}. In this article we numerically simulate a simplified case as shown in figure \ref{fig:simpcircuit}, a two-branch system with its lower branch firstly grounding and being cut off immediately after. Physically the longer the grounding state lasts, the longer it takes for the system to be stable again and there is a maximum time exceeding which the system cannot recover its stability, called the \emph{critical clearing time} (CCT). We will simulate the CCT and other corresponding physical quantities based on the electromagnetic transient model. 

Electromechanical transient model is a classic model to describe circuit systems\cite{PSSC}, which is easier to solve back to those eras without electronic computers since it presents the electric network by algebraic equations instead of differential equations and thus unsuitable for systems with high ratios of power electronic and renewable energy devices. Due to this, electromagnetic transient model is developed\cite{PMRA} and characters the  dynamic process of circuit systems accurately in continuous time level since it treats both the electric network and the mechanical part in form of differential equations equally. Electro-Magnetic Transient Program (EMTP) is a classical numerical method to solve this model. It divides the whole system into three subsystems, the circuit, the generator and the mechanical shaft, and exchange data among them after independent calculation of each subsystem. Thus EMTP is time consuming and lowers its numerical accuracy.

To improve this situation, Feng Ji et al. introduced predictor-corrector methods (P-C methods). This method firstly predicts the mechanical angle $\bm\theta$ by assuming that the angle speed constant in a time step $h$, i.e. $\bm{\theta}_{n+1}^{[0]}=\bm{\theta}_n+h\dot{\bm{\theta}}_n$. Then, this prediction $\bm{\theta}_{n+1}^{[0]}$ is used to predict the electric part $x_{E,n+1}=(\dot{\bm\Psi}_{n+1};\bm{\Psi}_{n+1})$ and to correct the mechanical part $x_{M,n+1}=(\dot{\bm\theta}_{n+1};\bm{\theta}_{n+1})$. See section \ref{sec:PCmethod} for details. P-C methods reveal a better numerical behaviour than EMTP but are long termly inferior to the structure-preserving method in Ref.~\cite{ENSS}. In this article, we will numerically compare P-C methods to the structure-preserving method in fault transient system.

Generally structure-preserving methods preserve a system's inherent structures and characteristic properties so that they usually exhibit better long term stability. Symplectic methods for Hamiltonian systems are typical structure-preserving methods first posed by Kang Feng\cite{S1} and then rapidly developed. For the electromagnetic transient model, by deeming the model a port-Hamiltonian system, Zhang et al.\cite{ENSS} proposed a method preserving a Dirac structure and applied it on electromagnetic steady state of synchronous generator system, resulting in long-term advantages. In this article we will apply this method on fault transient system, compare it to P-C methods and show that it also works for more general circuits in section \ref{sec:reducofgeneral}.

This article is organised as follows. In section \ref{sec:sgs} we briefly introduce synchronous generator system for fault transient case and derive the differential equations from Euler-Lagrange equation. Also we will show that the reduction process in Ref.~\cite{ENSS} also works for more general circuits in section \ref{sec:reducofgeneral}. Then, a special circuit of fault transient system will be given in section \ref{sec:specialcase} for later simulation. Sections \ref{sec:PCmethod} and \ref{sec:SPmethod} concisely recite constructions of P-C method and the structure-preserving method based on Dirac structure and port-Hamiltonian system. In section \ref{sec:simulation}, numerical simulations of P-C method and structure-preserving method are compared. Also the CCT and other corresponding physical quantities will be simulated and analysed for deeper understanding of fault transient systems. Finally in section \ref{sec:conclusion} a brief summery will be given.
\section{Synchronous generator system\label{sec:sgs}}
As shown in figure \ref{fig:ACSGS_tran}, a 7-winding generator is connected to an electric circuit of $n$ nodes, while the $n$-th node connects the generator directly and the first node touches the ground through a resistance $r_1\in (0,+\infty)$ in parallel to a Norton current source. For each $1\leq i<j\leq n$, nodes $i$ and $j$ are linked by a inductance $\ell_{ij}\in (0,+\infty]$, where $\ell_{ij}=+\infty$ simply represents that nodes $i$ and $j$ are not connected. Moreover, for $2\leq i\leq n$, the node $i$ may also touch the ground through a resistance $r_i\in (0,+\infty]$. Here again, if $r_i=+\infty$ then the circuit between node $i$ and the ground is open. According to Ref.~\cite{LMMS}, there are four circuit nodes in the 7-winding generator, i.e. nodes $f,\; D,\; g$ and $Q$.
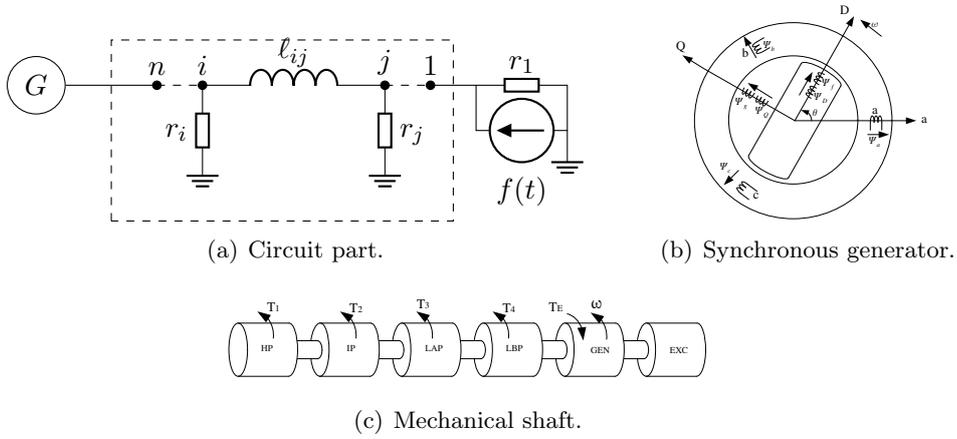
\begin{figure}[ht]
   \centering
   \subfigure[Circuit part.]{\begin{minipage}[t]{0.6\linewidth}
   \centering
   %\beginpgfgraphicnamed{main-general}
   \ctikzset{resistors/scale=0.4}
      \begin{tikzpicture}[scale=0.6,american]
        \draw (1,0) node[above]{$n$} to[short,*-] (-1,0) node[draw,circle,left]{$G$};
        \draw (2,-2) node[tlground]{} to[short,-*,european,R=$r_i$] (2,0) node[above]{$i$} to[short,-*,L=$\ell_{ij}$] (6,0) node[above]{$j$} to[european,R=$r_j$] (6,-2) node[tlground]{};
        \draw (7,0) node[above]{$1$} to[short,*-] (8,0) to[european,R=$r_1$] (10,0) -- (10,-1) node[ground]{} to[isource,l=$f(t)$] (8,-1) -- (8,0);
        \draw[dashed] (0,-3) rectangle (7.5,1);
        \draw[dashed] (1,0) -- (2,0);
        \draw[dashed] (6,0) -- (7,0);
      \end{tikzpicture}
    %\endpgfgraphicnamed

    \end{minipage}\label{fig:genecircuit}}
   \subfigure[Synchronous generator.]{\begin{minipage}[t]{0.3\linewidth}
   \centering
   \includegraphics[scale=0.3]{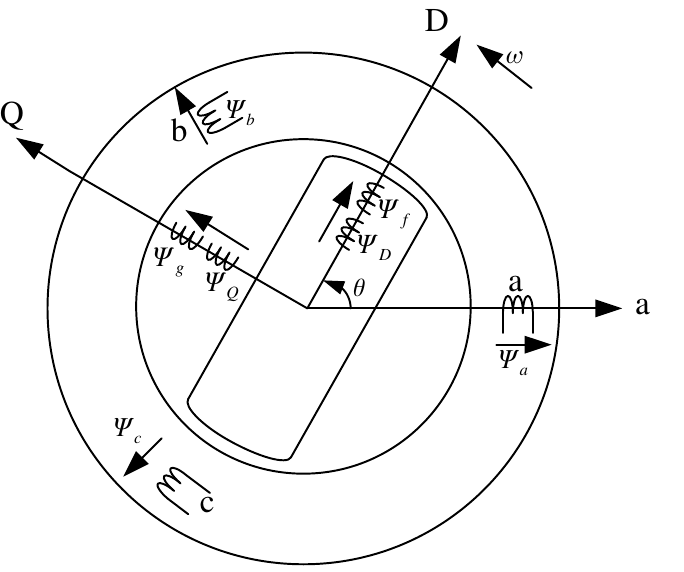}
   \end{minipage}}
   \subfigure[Mechanical shaft.]{\begin{minipage}[t]{0.5\linewidth}
   \centering
   \includegraphics[scale=0.4]{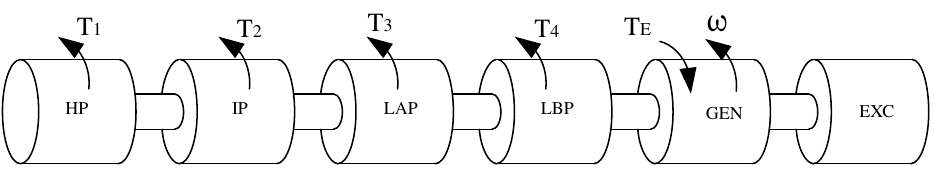}   \end{minipage}}
   
   \caption{Synchronous generator system.}
   \label{fig:ACSGS_tran}
\end{figure}

\subsection{Modeling of General Case \label{section 2.1}}
According to Refs.~\cite{LMMS} and \cite{DTP}, to describe a circuit by electromagnetic model, we use node flux linkages $\bm{\Psi}$ as position coordinates and node voltages $\dot{\bm{\Psi}}$ as velocity coordinates. In real world, electric circuits are 3-phase circuits of which each node in figure \ref{fig:ACSGS_tran} actually represents three nodes, such that they share a same module in voltages but the three phase angles differs $2\pi/3$ to each other. Here for simplicity, as suggested in Ref.~\cite{LMMS}, we use the $\alpha,\beta$ coordinate, i.e. each node $i$ ($1\leq i\leq n$) has two components $\alpha$ and $\beta$. For example, the flux linkages of node $1$ become $\Psi_{1\alpha},\Psi_{1\beta}$ now. Hence the flux linkages of the circuits can be written in vector form 
\[
    \bm{\Psi}=(\Psi_{1\alpha},\Psi_{1\beta},\dots,\Psi_{n\alpha},\Psi_{n\beta},\Psi_f,\Psi_D,\Psi_g,\Psi_Q)^\top\in \mathbb{R}^{2n+4},
\]
and $\dot{\bm{\Psi}}$ represents the voltages of the circuit.

In addition, there are six mass blocks in a 7-winding generator. It is natural to use
\[
    \bm{\theta}=(\theta_1,\theta_2,\theta_3,\theta_4,\theta_5,\theta_6)^\top
\]
and $\dot{\bm{\theta}}$ to represent the angular displacements and angular velocities of these blocks, respectively. 

It is shown\cite{LMMS,DTP} that the Lagrangian of this generator system has the following form 
\begin{equation}\label{eq:lag}
\mathcal{L}(\dot{\bm{\Psi}},\dot{\bm{\theta}},\bm{\Psi},\bm{\theta})
=\frac{1}{2}\dot{\bm{\theta}}^{\top} J\dot{\bm{\theta}}-\left(\frac{1}{2}\bm{\Psi}^{\top}(K_L+\Ga(\theta)) \bm{\Psi}+\frac{1}{2}\bm{\theta}^{\top} K\bm{\theta}\right)
\end{equation}
and the Rayleigh's dissipation function is 
\begin{equation}\label{eq:ray}
    \mathcal{R}=\frac{1}{2}\dot{\bm{\Psi}}^{\top}K_R \dot{\bm{\Psi}}+\frac{1}{2}\dot{\bm{\theta}}^{\top}D \dot{\bm{\theta}}-\dot{\bm{\theta}}^{\top}T-\dot{\bm{\Psi}}^{\top}f(t).
\end{equation}

For simplicity we list coefficient matrices in \eqref{eq:lag} and \eqref{eq:ray} directly in the following. One can consult Ref.~\cite{LMMS,DTP} for more physical meanings about them. First for the constant matrices,
\begin{align*}
    K_R&=\diag(r_1^{-1},r_1^{-1},\dots,r_n^{-1},r_n^{-1},r_f^{-1},r_D^{-1},r_g^{-1},r_Q^{-1})\in \m R^{2n+4},\\
    J&=\mathrm{diag}(J_1,J_2,J_3,J_4,J_5,J_6),\;D=0,\\
    T&=(T_1,T_2,T_3,T_4,0,0)^\top,\\
    K&=\begin{pmatrix}
    K_1 & -K_1 & 0 & 0 & 0 & 0 \\
    -K_1 & K_1+K_2	& -K_2 & 0 & 0 & 0 \\
    0 & -K_2 & K_2+K_3 & -K_3 & 0 & 0 \\
    0 & 0 & -K_3 & K_3+K_4 & -K_4 & 0 \\
    0 & 0 & 0 & -K_4 & K_4+K_5 & -K_5 \\
    0 & 0 & 0 & 0 & -K_5 & K_5
    \end{pmatrix}
\end{align*}
and 
\begin{equation}\label{eq:L}
  K_L=\begin{pmatrix}
  L & 0 \\ 0 & 0
  \end{pmatrix}\in \m R^{(2n+4)\times (2n+4)},\quad 
  L=\begin{pmatrix}
  L_{1\al,1\al}& L_{1\al,1\beta} & \cdots & L_{1\al,n\al} & L_{1\al,n\beta}\\
  L_{1\beta,1\al}& L_{1\beta,1\beta} & \cdots & L_{1\beta,n\al} & L_{1\beta,n\beta}\\
  \vdots & \vdots & \ddots & \vdots &\vdots\\
  L_{n\al,1\al}& L_{n\al,1\beta} & \cdots & L_{n\al,n\al} & L_{n\al,n\beta}\\
  L_{n\beta,1\al}& L_{n\beta,1\beta} & \cdots & L_{n\beta,n\al} & L_{n\beta,n\beta}
  \end{pmatrix}\in \m R^{2n\times 2n}
\end{equation}
where for all $1\leq i,j\leq n$, 
\[
  \quad L_{i\al,j\beta}=0,\quad L_{i\al,j\al}=L_{i\beta,j\beta}=\begin{cases}
    \sum_{k\neq i} \ell_{ik}^{-1}, & i=j,\\
    -\ell_{ij}^{-1}, & i\neq j.
  \end{cases}
\]
Also for those depend on time $t$ and angular $\theta$,
\[
    f(t)=\left(\frac{U_s}{r_1}\cos(\om_s t),\frac{U_s}{r_1}\cos(\om_s t),0,\dots,0,\frac{U_f}{r_f},0,0,0\right)^\top
\]
where $\om_s=120\pi$ and $U_s,U_f$ are given constant positive numbers. 
\[
    \Gamma(\theta)=\begin{pmatrix}
      0 & 0\\ 0 & P(\theta)\Gamma_0P(-\theta)
    \end{pmatrix}\in \m R^{(2n+4)\times (2n+4)},\quad 
    P(\theta)=\begin{pmatrix}
      \cos\theta & -\sin\theta & 0\\
      \sin\theta & \cos\theta & 0\\
      0 & 0 & I_4
    \end{pmatrix}\in \m R^{6\times 6}
\]
where $\Gamma_0\in \m R^{6\times 6}$ is a given constant positive definite matrix.  

By setting the generalised position $\bm q:=(\bm\Psi,\bm\theta)$ and momentum $\dot{\bm q}:=(\dot{\bm\Psi},\dot{\bm\theta})$, the Euler-Lagrange equation containing Rayleigh's dissipation function
$$\frac{\mathrm{d}}{\mathrm{d}t}\left(\frac{\partial \mathcal{L}}{\partial \dot{\bm{q}}}\right)-\frac{\partial \mathcal{L}}{\partial \bm{q}}+\frac{\partial \mathcal{R}}{\partial \dot{\bm{q}}}=0,$$
implies the dynamical equation of the generator system 
\begin{equation}\label{eq:oldsys}
\left\{\begin{aligned}
&K_R\dot{\bm{\Psi}}+(K_L+\Ga(\theta))\bm{\Psi}=f(t), \\
&J\ddot{\bm{\theta}}+D\dot{\bm{\theta}}+K\bm{\theta}+\frac{1}{2}\bm{\Psi}^{\top} \frac{d\Ga(\theta)}{d \theta} \bm{\Psi}=T,
\end{aligned}\right.
\end{equation}
here $\bm{\Psi}^{\top} \frac{d\Ga(\theta)}{d \theta} \bm{\Psi}:=\left(0,0,0,0,\bm{\Psi}^{\top} \frac{d\Ga(\theta)}{d \theta} \bm{\Psi},0\right)^{\top}$ is an abbreviation, since $\Ga(\theta)$ depends on $\theta:=\theta_5$ only. For later use, we give the following lemma.
\begin{lemma}\label{lem:Npostive}
  The matrix $K_L$ is non-negative definite and
  \[
      N(\theta):=K_L+\Gamma(\theta)
  \]
  is positive definite for all $\theta\in [0,2\pi)$.
\end{lemma}
\begin{proof}
  Recall that\cite{MA} for any diagonally dominant, symmetric matrix $A$, if all diagonal elements $a_{ii}\geq 0$, then $A\geq 0$. Moreover, if $A$ is strictly diagonally dominant and $a_{ii}>0$ for all $i$, then $A>0$. 

  Now by definition of $L$, easy to check it is diagonally dominant with positive diagonal elements, so $L\geq 0$. Certainly $\Ga(\theta)\geq 0$ for all $\theta$, so $N(\theta)\geq 0$.
  
  All left is to show $z^\top N(\theta)z=0$ implies $z=0$ for all $\theta$. Write 
  \[
      L=\begin{pmatrix}
        L_1 & L_2\\ L_2^\top & L_3
      \end{pmatrix},\quad L_1\in \m R^{(2n-2)\times (2n-2)},\quad  L_3\in \m R^{2\times 2},
  \]
  then $L_1$ is strictly diagonally dominant with positive diagonal elements, so $L_1>0$. Suppose 
  \[
    z^\top N(\theta)z=z^\top K_L z+z^\top\Ga(\theta)z =0,
  \]
  then $z^\top K_Lz=z^\top\Ga(\theta)z=0$. Suppose $z=(x;y)$ for $x\in \m R^{2n-2}$ and $y\in \m R^6$, then $z^\top \Ga(\theta)z=0$ implies $y=0$ and hence $z=(x;0)$. Now $ z^\top K_Lz=0$ implies $x=0$.
\end{proof}
\subsection{Reduction of General Case}\label{sec:reducofgeneral}
To solve the system \eqref{eq:oldsys}, similarly as in Ref.~\cite{ENSS}, notice that $K_R$ could be singular since $r_i^{-1}$ could be zero for some $2\leq i\leq n$. Denote by
\[
    \La_1:=\{2i-1,\; 2i: r_i=+\infty,\; 1\leq i\leq n\},\quad \La_2:=\{1,\dots,2n+4\}-\La_1,
\]
i.e. $\La_1$ consists of those indexes $1\leq k\leq 2n+4$ s.t. the $k$-th diagonal element of $K_R$ is zero. Since $0<r_1,r_f,r_D,r_g,r_Q<+\infty$, we have $\{1,2,2n+1,2n+2,2n+3,2n+4\}\subset \La_2$. 

For each $k\in \La_1$, easy to check that the $k$-th element of $f(t)$ is always zero, so the first equation of \eqref{eq:oldsys} shows that given any $\theta$, $\Psi_k$ is a linear combination of $\Psi_j: j\in \La_2$ and thus for each $1\leq i\leq 2n+4$, $\Psi_i$ is a linear combination of $\Psi_j: j\in \La_2$. Write these facts in form of matrices,
\[
    \bm{\Psi}(t)=A(\theta(t))\wt{\bm\Psi}(t),\quad \wt{\bm\Psi}:=(\Psi_k: k\in \La_2),\quad \forall t.
\]
\begin{remark}
  Though $A(\theta(t))$ is given by \eqref{eq:Atheta}, we will see later it is also a linear combination of $\sin\theta,\cos\theta,\sin 2\theta,\cos 2\theta$, making it easier to calculate.
\end{remark}

In the following we prove that \eqref{eq:oldsys} is equivalent to the reductive system
\begin{equation}\label{eq:newsys}
  \left\{\begin{aligned}
  &\wt{K}_R \dot{\wt{\bm{\Psi}}}+\wt{N}(\theta)\wt{\bm{\Psi}}=\wt{f}(t),\\     
  &J\ddot{\bm{\theta}}+D\dot{\bm{\theta}}+K\bm{\theta}+\frac{1}{2}\wt{\bm{\Psi}}^{\top} \frac{d \wt{N}(\theta)}{d {\theta}} \wt{\bm{\Psi}}=T,
  \end{aligned}
  \right.    
\end{equation}
where naturally if we write $K_R=\diag(K_{R,1},\dots,K_{R,2n+4})$ and $f=(f_1,\dots,f_{2n+4})^\top$, then 
\begin{align}\label{eq:Ntheta}
    \wt{K}_R:&=\diag(K_{R,k}: k\in \La_2), \notag \\
    \wt{f}(t):&=(f_k: k\in \La_2)^\top, \notag\\
    \wt{N}(\theta):&=A^\top(\theta)N(\theta)A(\theta).
\end{align}

\begin{lemma}\label{lem:Azero}
  Suppose that 
  \[
      N=\binom{N_1}{N_2}
  \]
  is positive definite and $N_j\in \m R^{n_j\times n}$ for $j=1,2$, then there is a unique $A_0\in \m R^{n_1\times n_2}$ s.t. $A=\dbinom{A_0}{I_{n_2}}$ and $N_1A=0$. Moreover, let $\wt{N}:=A^\top N A$, then $\wt{N}=N_2A$ and $\wt N>0$.
\end{lemma}
\begin{proof}
  Suppose $N_1=(N_{11},N_{12})$ where $N_{ij}\in \m R^{n_i\times n_j}$, then $N_{11}>0$ since $N>0$. Thus the equation $0=N_1A=N_{11}A_0+N_{12}$ has a unique solution. Certainly $\rank A=n_2$ and hence $\wt{N}>0$. Now 
  \[
      \wt N=\begin{pmatrix}
        A_0^\top & I_{n_2}
      \end{pmatrix}
      \cdot \binom{0}{N_2A}=N_2A.
  \]
\end{proof}
\begin{lemma}\label{lem:Ntheta}
  Suppose that
  \[
      N(\theta)=\binom{N_1(\theta)}{N_2(\theta)} 
  \]
  is positive definite for all $\theta\in\mathbb{R}$, $A(\theta)=\dbinom{A_0(\theta)}{I_{n_2}}$ and $\wt N(\theta)=A^\top(\theta)N(\theta)A(\theta)$ as determined in lemma \ref{lem:Azero}, then for all $x\in \m R^{n_2}$, let $y(\theta):=A(\theta)x$, we have 
  \begin{align*}
    y^\top(\theta) N(\theta) y(\theta)&=x^\top \wt N(\theta)x,\\
    y^\top(\theta) \frac{d N(\theta)}{d\theta} y(\theta)&=x^\top \frac{d \wt N(\theta)}{d\theta} x.
  \end{align*}
  for all $\theta$.
\end{lemma}
\begin{proof}
  The first equation is trivial and by taking differentiation on both sides,
  \[
      \left(\frac{dy}{d\theta}\right)^\top Ny+y^\top \frac{dN}{d\theta} y + y^\top N \frac{dy}{d\theta} = x^\top \frac{d\wt N}{d\theta} x.
  \]
  Since $y=Ax=\binom{*}{x}$, we see $\frac{dy}{d\theta}=\binom{*}{0}$, which gives 
  \begin{align*}
    y^\top N \frac{dy}{d\theta}&=\left(\frac{dy}{d\theta}\right)^\top Ny\\
    &=(*,0)\binom{N_1}{N_2}Ax\\
    &=(*,0)\binom{0}{N_2A}x=0.
  \end{align*}
\end{proof}
Without loss of generality, by re-numbering nodes $1,\dots,n$, in this section we assume
\[
\La_1=\{1,\dots,m\},\quad \La_2=\{m+1,\dots,2n+4\}.
\]
\begin{theorem}\label{thm:oldnew}
	The original system \eqref{eq:oldsys} is equivalent to the reductive system \eqref{eq:newsys}. That is to say, 
	\begin{itemize}
		\item if $(\bm{\Psi},\bm{\theta})$ solves \eqref{eq:oldsys}, then $\big(\wt{\bm{\Psi}}:=(\Psi_k:k\in \La_2)^\top, \bm{\theta}\big)$ solves \eqref{eq:newsys}. 
		\item conversely if $\big(\wt{\bm{\Psi}},\bm{\theta}\big)$ solves \eqref{eq:newsys}, then $\big(\Psi(t):=A(\theta(t))\wt{\bm{\Psi}}(t),\bm{\theta}\big)$ solves \eqref{eq:oldsys} for the function
		\begin{equation}\label{eq:Atheta}
			A(\theta)=\dbinom{-N_{\La_1,\La_1}^{-1} N_{\La_1,\La_2}}{I_{|\La_2|}},
		\end{equation}
		where $N_{\La_i,\La_j}:=(N_{k,\ell})_{k\in \La_i,\ell\in \La_j}$ is the sub-matrix of $N$ consisting of those rows in $\La_i$ and columns in $\La_j$ and $|\La_2|$ is the cardinal of $\La_2$.
	\end{itemize}  
\end{theorem}

\begin{proof}
  This is routine by noticing that $\Psi(\theta(t))=A(\theta(t))\wt{\Psi}(t)$ is equivalent to $\wt\Psi=(\Psi_k: k\in \La_2)^\top$ for some $A=\dbinom{A_0}{I_{2n+4-m}}$. Now treat $\wt{\Psi}$ and $\Psi$ as $x$ and $y$ in lemma \ref{lem:Ntheta}, respectively.
\end{proof}
So now our target is to solve system \eqref{eq:newsys}, which requires the calculation of $A(\theta)$ and $\wt N(\theta)$ for each $\theta$. By theorem \ref{thm:oldnew} that is to inverse $N_{\La_1,\La_1}(\theta)$ for each $\theta$, which would be time consuming as one can guess. The following observation accelerates this process.

\begin{lemma}
  Denote $N_{\La_i,\La_j}$ simply by $N_{ij}$, then each element of $A_0(\theta)=-N_{11}^{-1}(\theta)\cdot N_{12}(\theta)$ and $\wt N(\theta)$ is a linear combination of $\sin\theta,\;\cos\theta,\;\sin 2\theta,\;\cos 2\theta$.
\end{lemma}
\begin{proof}
  Suppose the node next to the generator (e.g. node $n$ in figure \ref{fig:genecircuit} and node $3$ in figure \ref{fig:simpcircuit}) is connected to ground by some resistance $r<+\infty$, then $N_{11},N_{12}$ are constant. So $A_0$ is also constant and
  \[
    \wt N(\theta)=\begin{pmatrix}
      I & A_0^\top
    \end{pmatrix}N\begin{pmatrix}
      A_0\\I
    \end{pmatrix}
  \]
  a linear combination of $\sin\theta,\dots,\cos 2\theta$. If $r=+\infty$, then by definition of $\Lambda_1$ and $\Lambda_2$, 
  \[
  \Ga(\theta)=\begin{pmatrix}
    0 & 0 & 0 & 0\\
    0 & E(\theta)\Ga_1E(-\theta) & E(\theta)\Ga_2 & 0\\
    0 & \Ga_2^\top E(-\theta) & \Ga_3 &0\\
    0 & 0 & 0 & 0
  \end{pmatrix},\quad E(\theta)=\begin{pmatrix}
    \cos\theta & -\sin\theta\\
    \sin\theta &\cos\theta
  \end{pmatrix},\quad \Ga_0=\begin{pmatrix}
    \Ga_1 & \Ga_2\\
    \Ga_2^\top &\Ga_3
  \end{pmatrix}
  \]
  for some constants $\Ga_1\in \m R^{2\times 2},\; \Ga_2\in \m R^{2\times 4}$ and $\Ga_3\in \m R^{4\times 4}$. Now by $N=K_L+\Ga(\theta)$ and definition of $L$, 
  \[
    N_{11}=\begin{pmatrix}
      A & B\\ B^\top & \la_0 I_2
    \end{pmatrix}+
    \begin{pmatrix}
      0 & 0\\0 &E(\theta)\Ga_0 E(-\theta)
    \end{pmatrix}
  \]
  for constants $\la_0>0$ and $A,B$ (here $A$ represents \emph{not} $\binom{A_0}{I}$). Direct computation gives 
 \begin{equation}\label{eq:Noneone}
      Q^\top N_{11}Q=\begin{pmatrix}
        A &  \\  & P(\theta)\Ga_0 P(-\theta)+ C
      \end{pmatrix},\quad Q=\begin{pmatrix}
        I & -A^{-1}B\\ & I
      \end{pmatrix},\quad 
      C=\la_0I_2-B^\top A^{-1} B.
  \end{equation}
  We show that $C=\la I_2$ for some $\la>0$. Actually, by definition of $L$ we have for some $\la_{ij}=\la_{ji}$
  \[
      A=\begin{pmatrix}
        \la_{11}I_2&\dots&\la_{1k}I_2\\
        \vdots & &\vdots\\
        \la_{k1}I_2 & \dots & \la_{kk}I_2
      \end{pmatrix},\quad B=\begin{pmatrix}
        \la_1I_2 \\
        \vdots \\
        \la_k I_2 
      \end{pmatrix},
    \]
    which implies that
    \[
      A^{-1}=\begin{pmatrix}
        \mu_{11}I_2&\dots&\mu_{1k}I_2\\
        \vdots & &\vdots\\
        \mu_{k1}I_2 & \dots & \mu_{kk}I_2
      \end{pmatrix},\quad 
      B^\top A^{-1}B=
        \sum_{i,j=1}^k \la_i\mu_{ij}\la_j I_2.
    \]
    So $\la=\la_0-\sum_{i,j=1}^k \la_i\mu_{ij}\la_j$. Inserting this into \eqref{eq:Noneone} gives 
    \[
    N_{11}^{-1}=\begin{pmatrix}
      I & -B^\top A^{-1} E(-\theta)\\ & E(-\theta)
    \end{pmatrix}\begin{pmatrix}
      A^{-1} & \\ & \la^{-1}I_2
    \end{pmatrix}\begin{pmatrix}
      I & \\ -E(\theta) A^{-1} B & E(\theta)
    \end{pmatrix}.
    \]
    Since $N_{12}=*+D$ for some constant matrix $*$ and
    \[
    D=\begin{pmatrix}
      0 & 0\\ E(\theta) \Ga_2 & 0
    \end{pmatrix},
    \]
    we see $A_0(\theta)=(-N_{11}^{-1}\cdot *) -N_{11}^{-1} D$ with the first term $N_{11}^{-1}\cdot *$ a combination of $\sin\theta,\dots,\cos2\theta$ and the second term being 
    \[
    N_{11}^{-1} D=\begin{pmatrix}
      -\la^{-1} B^\top A^{-1} E(\theta) \Ga_2 & 0\\
      \la^{-1} E(\theta) \Ga_2 & 0
    \end{pmatrix},
    \]
    merely a combination of $\sin\theta$ and $\cos\theta$. Now $\wt N=N_{22}-N_{12}^\top N_{11}^{-1} N_{12}$ with $N_{22}$ constant and 
    \begin{align*}
      N_{12}^\top N_{11}^{-1} N_{12} &= (*+D^\top) N_{11}^{-1} (*+D)\\
      &=*N_{11}^{-1}*+*N_{11}^{-1} D+ D^\top N_{11}^{-1} * + D^\top N_{11}^{-1} D,
    \end{align*}
    where 
    \[
        D^\top N_{11}^{-1} D = \begin{pmatrix}
          \la^{-1} \Ga_2^\top \Ga_2 & 0\\
          0 & 0
        \end{pmatrix}
    \]
    is constant. This shows $\wt N(\theta)$ is a combination of $\sin\theta,\dots,\cos2\theta$.
\end{proof}
Direct computation gives
\begin{theorem}\label{thm:fourier}
  If 
  \[
    B(\theta)=S_0+S_1\sin\theta+C_1\cos\theta+S_2\sin 2\theta+C_2\cos 2\theta
  \]
  for all $\theta\in\mathbb{R}$ and $S_0,\dots,C_2$ are all constant, then 
  \begin{align*}
    8S_0&=\sum_{k=0}^3 B(\frac{k\pi}{2})+B(-\frac{k\pi}{2}),\\
    S_1&=\frac{1}{2}B(\frac{\pi}{2})-\frac{1}{2}B(-\frac{\pi}{2}),\\
    C_1&=\frac{1}{2}B(0)-\frac{1}{4}B(\pi)-\frac{1}{4}B(-\pi),\\
    S_2&=\frac{1}{2}B(\frac{\pi}{4})-\frac{1}{2}B(-\frac{\pi}{4})-\frac{\sqrt{2}}{2}S_1,\\
    C_2&=\frac{1}{2}B(0)+\frac{1}{4}B(\pi)+\frac{1}{4}B(-\pi)-S_0.
   \end{align*}
\end{theorem}
\begin{remark}
  We can use \eqref{eq:Ntheta} and \eqref{eq:Atheta} to calculate values of $A$ and $\wt N$ at $\theta=0, \pm \pi/4, \pm \pi/2, \pm \pi, \pm 3\pi/2$ first then values at arbitrary $\theta$ by theorem \ref{thm:fourier}.
\end{remark}

\subsection{A Specific Case}\label{sec:specialcase}
\begin{figure}[ht]
  \centering
  \subfigure[A Circuit for Stage I,\;II and III.]{
   %\beginpgfgraphicnamed{main-special-0} 
   \ctikzset{resistors/scale=0.3}
    \begin{tikzpicture}[scale=0.55,american]
    \draw (-1,0) node[draw,circle,left]{$G$} to[short,-*] (0,0) node[above]{$3$} to[short,-*,L=$\ell_1$] (6,0) node[above]{$1$} -- (6,-2) to[short,-*,L=$\ell_2$] (3,-2) node[above]{$2$} to[L=$\ell_2$] (0,-2) -- (0,0);
    \draw (6,0) to[european,R=$r_1$] (8,0) -- (8,-2) node[ground]{} to[isource,l=$f(t)$] (6,-2);
    \draw (3,-2) to[european,R=$r_2$] (3,-3) node[tlground]{};
  \end{tikzpicture}
  %\endpgfgraphicnamed
  }
  \subfigure[Stage I, $r_2=+\infty$.]{
   %\beginpgfgraphicnamed{main-special-1} 
   \ctikzset{resistors/scale=0.3}
    \begin{tikzpicture}[scale=0.55,american]
    \draw (-1,0) node[draw,circle,left]{$G$} to[short,-*] (0,0) node[above]{$3$} to[short,-*,L=$\ell_1$] (6,0) node[above]{$1$} -- (6,-2) to[short,-*,L=$\ell_2$] (3,-2) node[above]{$2$} to[L=$\ell_2$] (0,-2) -- (0,0);
    \draw (6,0) to[european,R=$r_1$] (8,0) -- (8,-2) node[ground]{} to[isource,l=$f(t)$] (6,-2);
  \end{tikzpicture}
  %\endpgfgraphicnamed
  }

  \subfigure[Stage II, $r_2=0$.]{
   %\beginpgfgraphicnamed{main-special-2} 
   \ctikzset{resistors/scale=0.2}
    \begin{tikzpicture}[scale=0.55,american]
    \draw (-1,0) node[draw,circle,left]{$G$} to[short,-*] (0,0) node[above]{$3$} to[short,-*,L=$\ell_1$] (6,0) node[above]{$1$} -- (6,-2) to[short,-*,L=$\ell_2$] (3,-2) node[ground]{} to[L=$\ell_2$] (0,-2) -- (0,0);
    \draw (6,0) to[european,R=$r_1$] (8,0) -- (8,-2) node[ground]{} to[isource,l=$f(t)$] (6,-2);
    \node at (3,-2) [above]{$2$};
  \end{tikzpicture}
  %\endpgfgraphicnamed
  }
  \subfigure[Stage III, $\ell_2=+\infty$.]{
   %\beginpgfgraphicnamed{main-special-3} 
   \ctikzset{resistors/scale=0.3}
    \begin{tikzpicture}[scale=0.55,american]
    \draw (-1,0) node[draw,circle,left]{$G$} to[short,-*] (0,0) node[above]{$3$} to[short,-*,L=$\ell_1$] (6,0) node[above]{$1$} to[european,R=$r_1$] (8,0) -- (8,-2) node[ground]{} to[isource,l=$f(t)$] (6,-2) -- (6,0);
    \draw (3,-2) node[above]{$2$} to[short,*-] (3,-2) node[ground]{};
  \end{tikzpicture}
  %\endpgfgraphicnamed
  }
  \caption{A Special Case.}
  \label{fig:simpcircuit}
\end{figure}
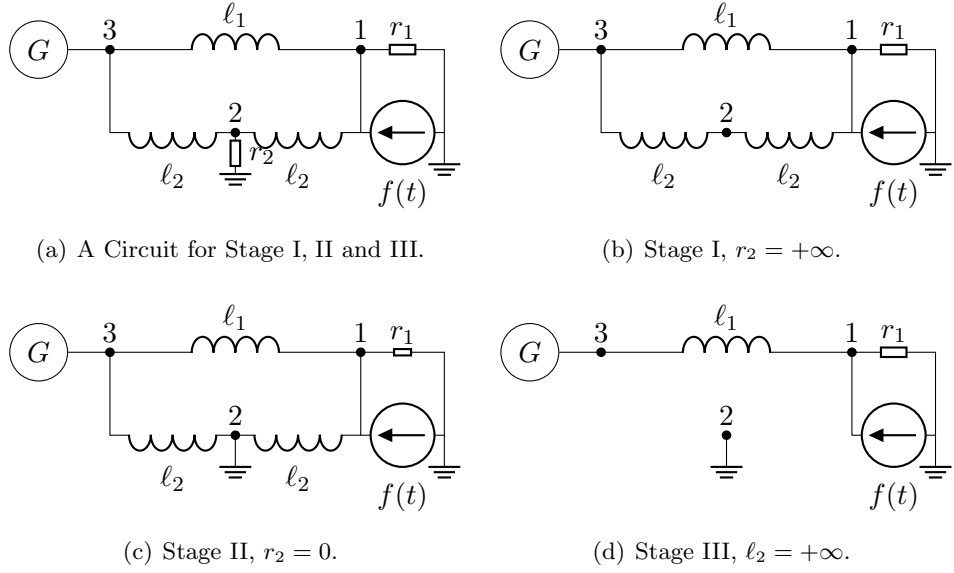
In this section we focus on a specific power system, which is a simplified example but complicated enough to model an electromagnetic transient system. As shown in figure \ref{fig:simpcircuit}, the power system was first running stably on stage I, then node $2$ is ground short and the system is transformed from stage I to II. After a short time $t_b$ (called the break time), usually less than one second, the lower branch would be removed leading the system to stage III. Each stage can be described by system \eqref{eq:oldsys} and the goal is to numerically simulate three stages, in order to find out the maximal possible $t_b$ such that stage III would be asymptotically stable. 

Rigorously speaking, node $2$ in stage II is connected to the ground through a resistance $r_2\to 0^+$ making the corresponding element $r_2^{-1}\to +\infty$ in the coefficient matrix $K_R$. In engineering numerical simulation, usually such ``ground resistance'' $r_2$ would be given a small value e.g. $1\ \mathrm{m\Omega}$. Similarly when removing the lower branch, two inductances $\ell_2$ tend to infinity and would become a large value, say $10^6\ \mathrm{H}$, while these extreme values could lead to numerical oscillation.

To avoid this, we rigorously treat $r_2=0$ and $\ell_2=+\infty$. The former makes $r_2^{-1}=+\infty$ in $K_R$ and hence the system \eqref{eq:newsys} of stage II not an ordinary differential equation (ode) of real number coefficients. An explanation is that if $r_i\to 0^+$ for some $2\leq i\leq n$, then the $i\al,i\beta$-th columns of $K_R \dot{\bm\Psi}+N(\theta)\bm\Psi=f(t)$ become
\[
    r_i^{-1} \dot{\Psi}_{i\gamma}+ N_{i\gamma} \bm{\Psi}=0,\quad \gamma=\al,\beta
\]
where $N_{i\al},\; N_{i\beta}$ are the $i\al,i\beta$-th column of $N(\theta)$ respectively. So $\dot{\Psi}_{i\gamma}=-r_iN_{i\gamma}\bm{\Psi}\to 0$ and thus we just set $\dot{\Psi}_{i\gamma}\equiv 0$. Hence $\Psi_{i\gamma}$ is a constant and we set $\Psi_{i\gamma}\equiv 0$ since the flux linkage of a ground short node should be zero. In this sense, we can simply remove the $i\alpha,i\beta$-th rows of $\bm{\Psi},\dot{\bm{\Psi}}$ and $f(t)$, the $i\al,i\beta$-th rows and columns of $K_R, N(\theta)$ and $dN(\theta)/d\theta$. For convenience, we confuse notations such that the rest part of $\bm{\Psi},\dot{\bm{\Psi}},\dots, N, dN/d\theta$ are still be denoted by $\bm{\Psi},\dot{\bm{\Psi}},\dots, N, dN/d\theta$ themselves respectively, where now $K_R$ becomes a real matrix and \eqref{eq:newsys} an ode of real number coefficients.

For the later part $\ell_2=+\infty$, it makes $L$ in \eqref{eq:L} having more than one zero eigenvalues so that the reduction method in section \ref{sec:reducofgeneral} fails. However, this will happen only if $\sum_{k\neq i} \ell_{ik}^{-1}=0$ for some $i$, i.e. $\ell_{ik}=+\infty$ for some $i$ and all $k\neq i$. This shows that node $i$ is free from all other nodes, which would happen only if it is ground short so that it would has been removed and $\wt{N}$ is invertible still. 

In the following we list values of all parameters of electrical components and all corresponding matrices. For those constant in Stage I to III:
The matrix $\Ga_0$ is given by 
\[
  \Ga_0=\begin{pmatrix}
    L_d & 0 & L_{df} & L_{dD} & 0 & 0\\
    0 & L_q & 0 & 0 & L_{qg} & L_{qQ}\\
    L_{fd} & 0 & L_f & L_{fD} & 0 & 0\\
    L_{Dd} & 0 & L_{Df} & L_D & 0 & 0\\
    0 & L_{gq} & 0 & 0 & L_g & L_{gQ}\\
    0 & L_{Qq} & 0 & 0 & L_{Qg} & L_Q
  \end{pmatrix}^{-1}
\]
where
\begin{align*}
  L_d &= 0.00359582836822968\\
  L_{df} = L_{fd} = L_{dD} = L_{Dd} &=0.0235797400806847\\
  L_q &= 0.00343512095512444\\
  L_{qg}=L_{gq}=L_{qQ}=L_{Qq}&= 0.0224433670647480\\
  L_f&=0.172961353354511\\
  L_{fD} = L_{Df} &= 0.166733941096683\\
  L_D &=0.167286372829233\\
  L_g &=0.191442705861614\\
  L_{gQ} = L_{Qg} &= 0.158698570441422\\
  L_Q &= 0.168240573094545
\end{align*}
The inertial matrix is 
\[
  J=\diag(J_1,J_2,J_3,J_4,J_5,J_6)
\] 
for
\begin{gather*}
J_1=1156.56,\; J_2=1953.83, \; J_3=10782.84,\\
J_4=11103.62,\; J_5=10906.22,\; J_6=429.68.
\end{gather*}
The stiff matrix is 
\[
  K=\begin{pmatrix}
    K_1 & -K_1 & 0 & 0 & 0 & 0 \\
    -K_1 & K_1+K_2	& -K_2 & 0 & 0 & 0 \\
    0 & -K_2 & K_2+K_3 & -K_3 & 0 & 0 \\
    0 & 0 & -K_3 & K_3+K_4 & -K_4 & 0 \\
    0 & 0 & 0 & -K_4 & K_4+K_5 & -K_5 \\
    0 & 0 & 0 & 0 & -K_5 & K_5
    \end{pmatrix}
\]
for 
\begin{gather*}
K_1=45692300.27,\;K_2=82680741.64,\\
K_3=123179695.3,\;K_4=167728592,\;K_5=6679980.902.
\end{gather*}
The total kinetic energy vector is 
\[
    T=(T_1,T_2,T_3,T_4,0,0)^\top
\]
for
\[
    T_1=0.3T_0,\;T_2=0.26T_0,\;T_3=T_4=0.22T_0,\;T_0=2130673.909092358.
\]
The resistances near node $1$ and inside the generator are given by 
\begin{align*}
  r_1&=5\times 10^{-4}\ \Omega\\
  r_f&=0.0532343305911098\ \Omega\\
  r_D&=0.154680885113791\ \Omega\\
  r_g&=0.532343305911098\ \Omega\\
  r_Q&=0.311370612891397\ \Omega
\end{align*}
The vector of injecting current is 
\[
    f(t)=\left(\frac{U_s}{r_1}\cos(\om_s t),\frac{U_s}{r_1}\cos(\om_s t),0,\dots,0,\frac{U_f}{r_f},0,0,0\right)^\top
\]
with $U_s=2.6\times 10^4\ \mathrm{V},\; \om_s=120\pi\ \mathrm{rad}/\mathrm{s},\; U_f=373.7756\ \mathrm{V}$.

For those changing from Stage I to III:
The resistance matrices are  
\begin{gather*}
  K_R^{\mathrm{I}}=\diag(r_1^{-1},r_1^{-1},0,0,0,0,r_f^{-1},r_D^{-1},r_g^{-1},r_Q^{-1})\\
  K_R^{\mathrm{II}}=K_R^{\mathrm{III}}=\diag(r_1^{-1},r_1^{-1},+\infty,+\infty,0,0,r_f^{-1},r_D^{-1},r_g^{-1},r_Q^{-1})
\end{gather*}
where the head indexes represent the stage I, II and III respectively. The inductance matrices are given by $K_L^{\mathrm{I}}=K_L^{\mathrm{II}}=\begin{pmatrix}
    L^{\mathrm{I}} & 0\\0 & 0
  \end{pmatrix}$,
\[
L^{\mathrm{I}}=L^{\mathrm{II}}=\begin{pmatrix}
  \ell_1^{-1}+\ell_{2}^{-1} & 0 & -\ell_2^{-1} & 0 &-\ell_1^{-1} & 0\\
  0 & \ell_1^{-1}+\ell_2^{-1} & 0 &-\ell_2^{-1} & 0 & -\ell_1^{-1}\\
  -\ell_2^{-1} & 0 & 2\ell_2^{-1} & 0 & -\ell_2^{-1} & 0\\
  0 & -\ell_2^{-1} & 0 & 2\ell_2^{-1} & 0 & -\ell_2^{-1}\\
  -\ell_1^{-1} & 0 & -\ell_2^{-1} & 0 & \ell_1^{-1}+\ell_2^{-1} & 0\\
  0 & -\ell_1^{-1} & 0 & -\ell_2^{-1} & 0 & \ell_1^{-1}+\ell_2^{-1} 
\end{pmatrix}
\]
and 
\[
K_L^{\mathrm{III}}=\begin{pmatrix}
  L^{\mathrm{III}} & 0\\0 & 0
\end{pmatrix},\quad
L^{\mathrm{III}}=\begin{pmatrix}
  \ell_1^{-1} & 0 & 0 & 0 &-\ell_1^{-1} & 0\\
  0 & \ell_1^{-1}& 0 & 0 & 0 & -\ell_1^{-1}\\
  0 & 0 & 0 & 0 & 0 & 0\\
  0 & 0 & 0 & 0 & 0 & 0\\
  -\ell_1^{-1} & 0 & 0 & 0 & \ell_1^{-1} & 0\\
  0 & -\ell_1^{-1} & 0 & 0 & 0 & \ell_1^{-1}
\end{pmatrix}
\]
where $\ell_1=4\times 10^{-4}\ \mathrm{H},\;\ell_2=2\times 10^{-4}\ \mathrm{H}$.

As stated in beginning, the system was running stably on stage I, mathematically speaking, on its equilibrium operating point. However, the first equation of \eqref{eq:oldsys} is time-dependent so we need to transform it to a autonomous form. Actually\cite{LMMS}, \eqref{eq:oldsys} is transformed from the xy synchronous coordinate system
\begin{equation}\label{eq:xy}
\left\{
  \begin{aligned}
    &K_R\dot{\bm\vphi}+(K_L+\om_sK_jK_R+\Ga(\delta))\bm\vphi=f_0,\\
    &J\ddot{\bm\delta}+D\dot{\bm\delta}+K\bm\delta+\frac{1}{2}\bm{\vphi}^\top \frac{d\Ga(\delta)}{d\delta}\bm{\vphi}+D\bm{\om}_s=T
  \end{aligned} 
\right.
\end{equation}
where $f_0=f(0)=(U_s/r_1,0,0,\dots,0,U_f/r_f,0,0,0)^\top$, $\bm{\om}_s=\om_s(1,1,1,1,1,1)^\top$,
\[
  K_j=\diag\left(\begin{pmatrix}
    0 & -1 \\1 & 0
  \end{pmatrix},\dots,\begin{pmatrix}
    0 & -1 \\1 & 0
  \end{pmatrix},0,0,0,0\right)
\]
and still $\bm{\delta}=(\delta_1,\dots,\delta=\delta_5,\delta_6)^\top$ and $\bm{\vphi}^\top \frac{d\Ga(\delta)}{d\delta}\bm{\vphi}=(0,0,0,0,\bm{\vphi}^\top \frac{d\Ga(\delta)}{d\delta}\bm{\vphi},0)^\top$ is an abbreviation. In detail, fix a time $t$, \eqref{eq:xy} is transformed into \eqref{eq:oldsys} by 
\begin{equation}\label{eq:trans}
  \bm{\theta}=\bm{\delta}+t\bm{\om}_s,\quad \binom{\Psi_{i\al}}{\Psi_{i\beta}}=\begin{pmatrix}
    \cos(\om_s t) & -\sin(\om_s t)\\
    \sin(\om_s t) & \cos(\om_s t)
  \end{pmatrix}\binom{\vphi_{ix}}{\vphi_{iy}},\quad 1\leq i\leq n
\end{equation}
where $\bm{\vphi}=(\vphi_{1x},\vphi_{1y},\dots,\vphi_{nx},\vphi_{ny},\Psi_f,\Psi_D,\Psi_g,\Psi_Q)^\top$. 

For later use, we point out that if \eqref{eq:xy} is applied to stage III then there is a different equilibrium point (which is $47.421^\circ$ numerically) compared to that of stage I. Actually this $47.421^\circ$ represents the angle value to which the transient system should converge if it is finally stable.
\section{A Predictor-Corrector Method\label{sec:PCmethod}}
For completeness in this section we briefly recite a predictor-corrector method (P-C method), which is based on the actual physical phenomenon and is an explicit numerical method with relatively long stability.\cite{ENSS} We would compare this P-C method to our structure-preserving method mentioned next section numerically.

The system \eqref{eq:oldsys} is equivalent to 
\begin{equation}\label{eq:sysPC}
\left\{
  \begin{aligned}
  K_{E_1}\dot{x}_E&=-K_{E_2}(\bm\theta)x_E+g_E(t),\\     
  K_{M_1}\dot{x}_M&=-K_{M_2}x_M+g_M(\bm{\Psi},\bm\theta),
  \end{aligned}
\right.     
\end{equation}
where $x_E=(\dot{\bm{\Psi}}; \bm{\Psi}),\ x_M=(\dot{\bm{\theta}}; \bm{\theta})$ and 
\begin{align*}
\begin{split}
&K_{E_1}=\diag(0,I_{2n+4}),\ K_{E_2}(\bm\theta)
                            =\begin{pmatrix}
                            K_R & K_L+\Gamma(\theta) \\
                            -I_{2n+4} & 0 
                            \end{pmatrix},\      
g_E(t)=\begin{pmatrix} f(t) \\ 0 \end{pmatrix}, \\
&K_{M_1}=\mathrm{diag}(J,I_6),\ K_{M_2}=\left(\begin{array}{cc}
                                           D & K \\
                                           -I_6 & 0 
                                           \end{array}\right),\ 
g_M(\bm{\Psi},\bm{\theta})=\begin{pmatrix}
T-\frac{1}{2}\bm{\Psi}^{\top} \frac{d \Gamma(\theta)}{d \theta} \bm{\Psi} \\ 0 \end{pmatrix}.
\end{split}
\end{align*}
Here the script $E$ represents electromagnet and $M$ mechanic. 

Since in reality the mechanical part $\bm\theta$ changes much slower than the electromagnetic part $\bm\Psi$, 
\[
    \bm{\theta}_{n+1}^{[0]}=\bm{\theta}_n+h\dot{\bm{\theta}}_n
\]
is used to predict angles in time $t_{n+1}=(n+1)h$. With this, applying the finite difference scheme with coefficient $\beta\in[0,1]$ to the first equation of \eqref{eq:sysPC} reads
\begin{align*}
  &\left[K_{E_1}+\beta h K_{E_2}(\bm\theta_{n+1}^{[0]})\right]x_{E,n+1} \\
  &=\left[K_{E_1}-(1-\beta)h K_{E_2}(\bm\theta_n)\right]x_{E,n}+h\left((1-\beta)g_E(t_n)+\beta g_E(t_{n+1})\right).  
\end{align*}
Similarly, to the second equation of \eqref{eq:sysPC} yields
\begin{align*}
  &\left(K_{M_1}+\beta h K_{M_2}\right)x_{M,n+1} \\
  &=\left(K_{M_1}-(1-\beta) h K_{M_2}\right)x_{M,n}+h\left((1-\beta) g_M\left(\bm{\Psi}_n,\bm{\theta}_n\right)+\beta g_M\left(\bm{\Psi}_{n+1},\bm{\theta}_{n+1}^{[0]}\right)\right),
\end{align*}  
in which $(\dot{\bm{\theta}}_n;\bm{\theta}_{n+1}^{[0]})$ is ``corrected'' by $x_{M,n+1}$.

\section{Structure-Preserving Methods \label{sec:SPmethod}}
\subsection{Dirac Structure}
Let 
\[
x(t)=\begin{pmatrix} \dot{\wt{\bm{\Psi}}} \\ \wt{\bm{\Psi}} \\ \dot{\bm{\theta}} \\ \bm{\theta} \\ t \end{pmatrix},\ u(x)=\begin{pmatrix} \wt{f}(t) \\ 0 \\ T \\ 0 \\ 1
\end{pmatrix},\ y(x)=\begin{pmatrix}
\dot{\wt{\bm{\Psi}}} \\ 0 \\ \dot{\bm{\theta}} \\ 0 \\ 0 \end{pmatrix} 
\]
then \eqref{eq:newsys} can be written as 
\begin{equation}\label{eq:sysPH}
\begin{aligned}
M\dot{x}&=(P-Q)z(x)+(F-V)u(x), \\
y(x)&=(F+V)^{\top}z(x)+(S-W)u(x),
\end{aligned}
\end{equation}
where the coefficient matrices are given by
\[
  \begin{split}
  &M=\diag\left(0,I,J,I,1\right),\ F=\mathrm{diag}(I,0,I,0,1),\ V=S=W=0, \\  
  &P=\diag\left(\begin{pmatrix}
  0 & -I \\ I & 0
  \end{pmatrix},\begin{pmatrix}
    0 & -I \\ I & 0
    \end{pmatrix},0\right),\ Q=\mathrm{diag}\left(\wt{K}_R,0,D,0,0\right)
  \end{split}
\]
and 
\[
z(x)=\left(\begin{array}{c}
\dot{\wt{\bm{\Psi}}} \\ \wt{N}(\theta)\wt{\bm{\Psi}} \\
\dot{\bm{\theta}} \\
K\bm{\theta}+\frac{1}{2}\wt{\bm{\Psi}}^{\top} \frac{d \wt{N}(\theta)}{d \theta} \wt{\bm{\Psi}} \\ 0
\end{array}\right).
\]
By setting the skew-symmetric matrix $A$ and the non-negative definite matrix $B$ as 
\[
A:=\begin{pmatrix}
  P & F \\-F^\top & W
\end{pmatrix},\ 
B:=\begin{pmatrix}
  Q & V \\ V^\top & S
\end{pmatrix}
\]
system \eqref{eq:sysPH} is an autonomous port-Hamiltonian system according to Ref.~\cite[Definition 1]{pHDAE}, with a Dirac structure. 

\begin{definition}\label{Def. linear Dirac structure}
  Let $\mathcal{F}$ be an $n$-dimensional linear space and $\mathcal{E}=\mathcal{F}^{\ast}$ its dual space. In addition, $\mathcal{U}$ is another linear space of dimension $n$, $F,E$ are $n\times n$ matrices representing the linear maps $F:\mathcal{F}\to \mathcal{U}$ and $E:\mathcal{E}\to \mathcal{U}$, respectively. Therefore, a linear subspace
  \[
  \mathcal{D}=\left\{(v_f,v_e)\in \mathcal{F}\times \mathcal{E}\ |\ Fv_f+Ev_e=0\right\}\subseteq \mathcal{F}\times \mathcal{E}    
  \]
  is a Dirac structure, if the matrices $F,E$ satiesfy
  \[
  \begin{split}
  &(i)\  EF^{\top}+FE^{\top}=0, \\
  &(ii)\ \mathrm{rank}(F, E)=n.
  \end{split}
  \]
  \end{definition}
  \begin{definition}\label{Def. general Dirac structure}
  Let $\mathcal{X}$ be a manifold and $\mathcal{V}$ be a vector bundle over $\mathcal{X}$ with fibers $\mathcal{V}_x\ (x\in \mathcal{X})$. A Dirac structure on $\mathcal{V}$ is a vector sub-bundle $\mathcal{D}\subseteq \mathcal{V}\oplus \mathcal{V}^{\ast}$ such that
  $$\mathcal{D}_x\subseteq \mathcal{V}_x\oplus \mathcal{V}_x^{\ast}$$
  is a linear Dirac structure for every $x\in \mathcal{X}$. Here $\mathcal{V}^*$ is the dual bundle of $\mathcal{V}$ and $\oplus$ represents Whitney sum.
\end{definition}
Now let $\mathcal{X}=\m R^m$ with $x(t)\in \m R^m$ and 
\[
  \mathcal{V}:=MT\mathcal{X}\oplus \ep^m\oplus \ep^{2m}
\]
where $T\mathcal{X}$ is the tangle bundle and $\ep^m=\mathcal{X}\times \m R^m$ the trivial bundle. Also, for all $p\in \mathcal{X}=\m R^m$, define
\[
    \mathcal{D}_{p}=\left\{(v_f,v_e)\in \mathcal{V}_{p}\oplus \mathcal{V}_{p}^{\ast}\ \Bigg|\ v_f+\begin{pmatrix}
       A & I_{2m}  \\
       -I_{2m} & 0 
    \end{pmatrix}v_e=0\right\},  
\]
then the sub-bundle $\mathcal{D}$ with fiber $\mathcal{D}_p$ is a Dirac structure on $\mathcal V$. Easy to check that if $x:\m R\to \m R^m$ solves \eqref{eq:sysPH} then 
\[
  (v_f(t),v_e(t))\in \mathcal{D}_{x(t)},\quad \forall t\in \m R 
\]
where 
\begin{equation}\label{eq:contife}
    v_f(t)=\begin{pmatrix}
      -M\dot x(t)\\
      y(x(t))\\
      z(x(t))\\
      u(x(t))
    \end{pmatrix},\quad 
    v_e(t)=\begin{pmatrix}
      z(x(t))\\
      u(x(t))\\
      -B\binom{z(x(t))}{u(x(t))}
    \end{pmatrix}.
\end{equation}
This shows that \eqref{eq:sysPH} has the Dirac structure $\mathcal D$.

Numerically we apply an $s$-stage Runge-Kutta method 
\begin{align}\label{eq:RK}
\begin{split}
Mk_i&=(P-Q)z\left(x_0+h\sum\limits_{j=1}^s a_{ij}k_j\right)+(F-V)u\left(x_0+h\sum\limits_{j=1}^s a_{ij}k_j\right), \\
x_f&=x_0+h\sum\limits_{j=1}^s b_j k_j.
\end{split}
\end{align}
Consequently, there exists a discrete Dirac structure $\left\{\mathcal{D}_{x_i}\ |\ i=1,\cdots,s\right\}$ defined by 
\begin{align}\label{eq:disDirac}
\mathcal{D}_{x_i}=\left\{(v_{f,i},v_{e,i})\in \mathcal{V}_{x_i}\oplus \mathcal{V}_{x_i}^{\ast}\ \Bigg|\ v_{f,i}+\begin{pmatrix}
   A & I_{2m}  \\
   -I_{2m} & 0 
\end{pmatrix}v_{e,i}=0\right\}    
\end{align}
at all points $x_i:=x_0+h\sum_{j=1}^s a_{ij}k_j$. Similarly as the continuous case \eqref{eq:contife} for each $i$ define 
\[
  v_{f,i}=\begin{pmatrix}
    -M k_i\\
    y(x_i)\\
    z(x_i)\\
    u(x_i)
  \end{pmatrix},\quad 
  v_{e,i}=\begin{pmatrix}
    z(x_i)\\
    u(x_i)\\
    -B\binom{z(x_i)}{u(x_i)}
  \end{pmatrix},
\]
then the first equation of \eqref{eq:RK} is equivalent to $(v_{f,i},v_{e,i})\in \mathcal{D}_{x_i}$ for all $i=1,\dots,s$. This shows that the method \eqref{eq:RK} obeys the discrete Dirac structure \eqref{eq:disDirac}.

We list two Runge-Kutta methods 
\begin{table}[ht]
	\caption{Two Runge-Kutta Methods}
	\renewcommand\arraystretch{1.5}
	\centering
	\subtable{
		\begin{tabular}{c|c}
			$c$ & $A$ \\
			\hline
			&    $b^\top$
		\end{tabular}\quad =
	}
	\subtable[]{\label{rk:impeuler}
		\centering
		\begin{tabular}{c|c}
			$1$ & $1$ \\
			\hline
			&    1
		\end{tabular}\quad or
	}
	\subtable[]{\label{rk:impmid}
		\centering
		\begin{tabular}{c|c}
			$\frac{1}{2}$ & $\frac{1}{2}$ \\
			\hline
			&    1
		\end{tabular}
	}
\end{table}
with coefficient matrices\footnote{The vector $c=(c_i)$ of a Runge-Kutta method is not used for autonomous differential equations.} $A=(a_{ij}),b=(b_j)$. Here table \ref{rk:impeuler} represents the implicit Euler method of order one and \ref{rk:impmid} for the implicit midpoint method of order two.

\subsection{Switching between two Stages}
Suppose that the system operates on stage I in time $[0,t_1]$ and at time $t_1$ node $2$ is suddenly short ground so that stage I transforms to stage II. Also suppose that at $t_1$, the system has state quantity $\bm{\Psi}^{\mathrm I} (t_1),\; \dot{\bm\Psi}^{\mathrm I}(t_1),\; \bm{\theta}^{\mathrm I}(t_1)$ and $\dot{\bm{\theta}}^{\mathrm I}(t_1)$, we have to transform it to the initial condition of stage II, i.e. to decide $\bm{\Psi}^{\mathrm{II}} (t_1),\; \dot{\bm\Psi}^{\mathrm{II}}(t_1),\; \bm{\theta}^{\mathrm{II}}(t_1)$ and $\dot{\bm{\theta}}^{\mathrm{II}}(t_1)$.

By Newton Law II and the fact that electromagnetic part changes much faster than mechanical part, it is reasonable to assume that mechanical part keeps unchanged in a short time, i.e. 
\begin{equation}\label{eq:thetala}
	\bm{\theta}^{\mathrm{II}}(t_1)=\bm{\theta}^{\mathrm I}(t_1),\quad \dot{\bm{\theta}}^{\mathrm{II}}(t_1)=\dot{\bm{\theta}}^{\mathrm I}(t_1).
\end{equation}
As for flux linkages and voltages, we define for stage II
\begin{gather*}
    \Lambda_0:=\{2i-1,2i: r_i=0\},\\
	\Lambda_1:=\{2i-1,2i: r_i=+\infty\},\; \Lambda_2:=\{2i-1,2i: 0<r_i<+\infty\},
\end{gather*}
i.e. $\Lambda_1$ represents those nodes totally disconnected to ground in stage II, $\Lambda_2$ those nodes connected to ground by some resistance $0<r_i<+\infty$ and $\Lambda_0$ those nodes short ground. Then by the assumption that flux linkages and voltages of nodes short ground should be zero (see beginning of section \ref{sec:specialcase}), we see
\begin{equation}\label{eq:psila3}
	\Psi_k^{\mathrm{II}}(t_1)=0,\quad \forall k\in \Lambda_0.
\end{equation}
For $k\in \Lambda_2$, since $0<r_k<+\infty$ and $r_k \dot{\Psi}_k +N_k(\theta)\bm\Psi=f_k(t)$, here $N_k,\; f_k$ are the $k$-th row of $N,\; f$ respectively, we know $|\dot{\Psi}_k(t_1)|<+\infty$ and hence $\Psi_k$ keeps unvaried in a short time:
\begin{equation}\label{eq:psila2}
	\Psi_k^{\mathrm{II}}(t_1)=\Psi_k^{\mathrm{I}}(t_1),\quad \forall k\in \Lambda_2.
\end{equation}
Finally for $k\in \Lambda_1$, by theorem \ref{thm:oldnew} and letting $\bm\Psi_{\Lambda_j}:=(\Psi_k: k\in \Lambda_j)$, 
\begin{equation}\label{eq:psila1}
	\bm\Psi_{\Lambda_1}^{\mathrm{II}}=A_0 (\theta(t_1))\cdot \bm\Psi_{\Lambda_2}^{\mathrm{II}} (t_1),\quad A_0=-N_{\Lambda_1,\Lambda_1}^{-1}\cdot N_{\Lambda_1,\Lambda_2}.
\end{equation}

Now for voltages $\dot{\bm{\Psi}}^{\mathrm{II}}(t_1)$. If $k\in \Lambda_0$, then beginning of section \ref{sec:specialcase} shows 
\begin{equation}\label{eq:dotpsila3}
  \dot{\Psi}_k^{\mathrm{II}}(t_1)=0,\quad \forall k\in \Lambda_0.
\end{equation}
If $k\in \Lambda_2$, then $r_k \dot{\Psi}_k +N_k(\theta)\bm\Psi=f_k(t)$ gives
\begin{equation}\label{eq:dotpsila2}
  \dot{\Psi}_k^{\mathrm{II}}(t_1)=r_k^{-1}\Big(f_k(t_1)-N_k (\theta(t_1))\cdot\bm{\Psi}^{\mathrm{II}}(t_1)\Big).
\end{equation}
If $k\in \Lambda_1$, then \eqref{eq:psila1} implies 
\begin{equation}\label{eq:dotpsila1}
  \begin{aligned}
    \dot{\bm\Psi}_{\Lambda_1}(t_1)=
    &\frac{d}{dt}\Big(A_0 (\theta(t_1))\cdot \bm\Psi_{\Lambda_2}^{\mathrm{II}} (t_1)\Big)\\
    =&\Big(\frac{d}{dt}A_0 (\theta(t_1))\Big)\cdot \bm\Psi_{\Lambda_2}^{\mathrm{II}} (t_1)\\
    &+A_0 (\theta(t_1))\cdot \dot{\bm\Psi}_{\Lambda_2}^{\mathrm{II}} (t_1).
  \end{aligned}
\end{equation}

In summery, this structure-preserving method obeys the following processes.
\begin{itemize}
  \item[\textbf{S1}] To solve the equilibrium point of \eqref{eq:xy} and transform it to the initial condition of stage I by \eqref{eq:trans}, this is called the stable point of stage I.
  \item[\textbf{S2}] To reduce stage I to system \eqref{eq:newsys} and apply method \eqref{eq:RK} with coefficient table \ref{rk:impeuler} or \ref{rk:impmid} to \eqref{eq:sysPH}, which is equivalent to \eqref{eq:newsys}.
  \item[\textbf{S3}] To determine the initial condition of stage II by \eqref{eq:thetala} to \eqref{eq:dotpsila1}.
  \item[\textbf{S4}] To repeat \textbf{S2} and \textbf{S3} for stage II and stage III.
  \item[\textbf{S5}] To compare the end period value (generated by \textbf{S2} on stage III) and the stable point (generated by \textbf{S1} on stage III) of stage III to decide whether the system recover its stability and how this numerical method performs.
  \captionof{process}{Steps of the Structure-Preserving Method.}
  \label{pro:SPmethod}
\end{itemize}
\section{Numerical Simulations\label{sec:simulation}}
By solving the equilibrium point of \eqref{eq:xy} for stage I and inserting it into \eqref{eq:trans} at time $t=0$, we get the initial condition:
\begin{align*}
\dot{\bm{\Psi}}_0=\left(\begin{array}{c}
  26015.4363\\
  -6.3200\\
  26491.9549\\
  1157.5512\\
  26968.4734\\
  2321.4224\\
  0\\
  0\\
  0\\
  0
\end{array}\right), \
\bm{\Psi}_0=\left(\begin{array}{c}
  -0.0168\\
-69.0081\\
3.0705\\
-70.2721\\
6.1578\\
-71.5361\\
492.6430\\
448.9184\\
-297.6778\\
-297.6778
\end{array}\right),\ \dot{\bm{\theta}}_0=\begin{pmatrix}
  120\pi \\ 120\pi \\ 120\pi \\ 120\pi \\ 120\pi \\ 120\pi
\end{pmatrix},\ \bm{\theta}_0=\left(\begin{array}{c}
  -0.7429\\
-0.7569\\
-0.7713\\
-0.7848\\
-0.7975\\
-0.7975
\end{array}\right),   
\end{align*}
here all data (except $0,120\pi$) are account to four decimal places. 

The P-C method with $\beta=1$ and the structure preserving method with coefficient matrices in table \ref{rk:impeuler} are both of order one. Similarly the P-C method with $\beta=0.5$ and the structure preserving method with coefficient matrices in table \ref{rk:impmid} are both of order two. We apply all these four methods to the case \ref{sec:specialcase} with break time $t_b=0.5\ \mathrm{s}$, respectively, and the results are shown in figure \ref{fig:tbreak5}. It is shown that for both order one and order two numerical methods, our structure preserving methods behave better than P-C methods. The $5$-th angular velocity $\omega_5$ tends to $120\pi\ \mathrm{rad}/\mathrm{s}$ quicker by our structure preserving methods.
\begin{figure}[ht]
 \centering
 \subfigure[1st order methods in $0\sim 11\ \mathrm{s}.$]{
 	\includegraphics[scale=0.2]{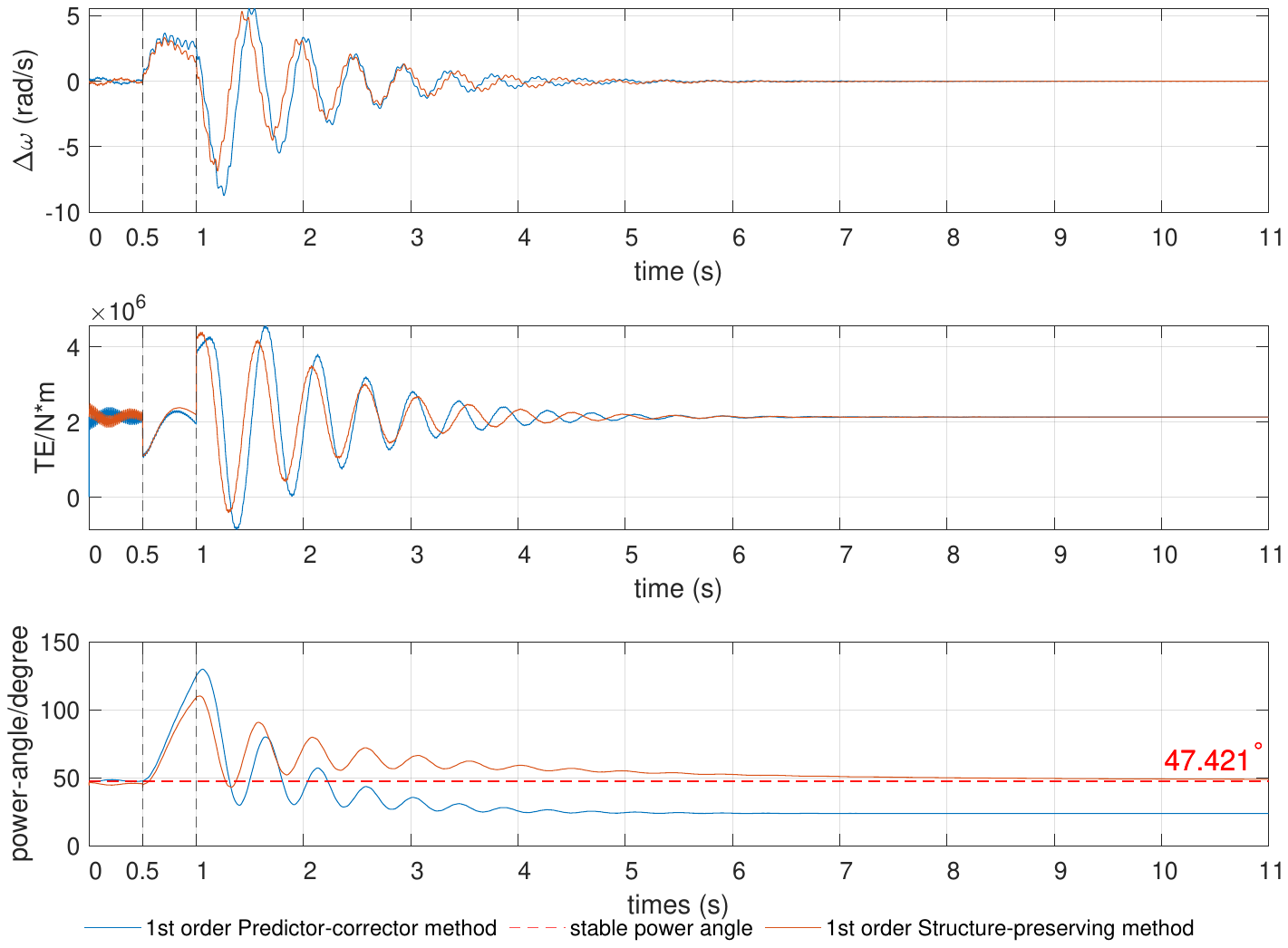}
 	\label{fig:Gauss0pt5}
 }
 \subfigure[1st order methods in $50\sim 100\ \mathrm{s}.$]{
 	\includegraphics[scale=0.2]{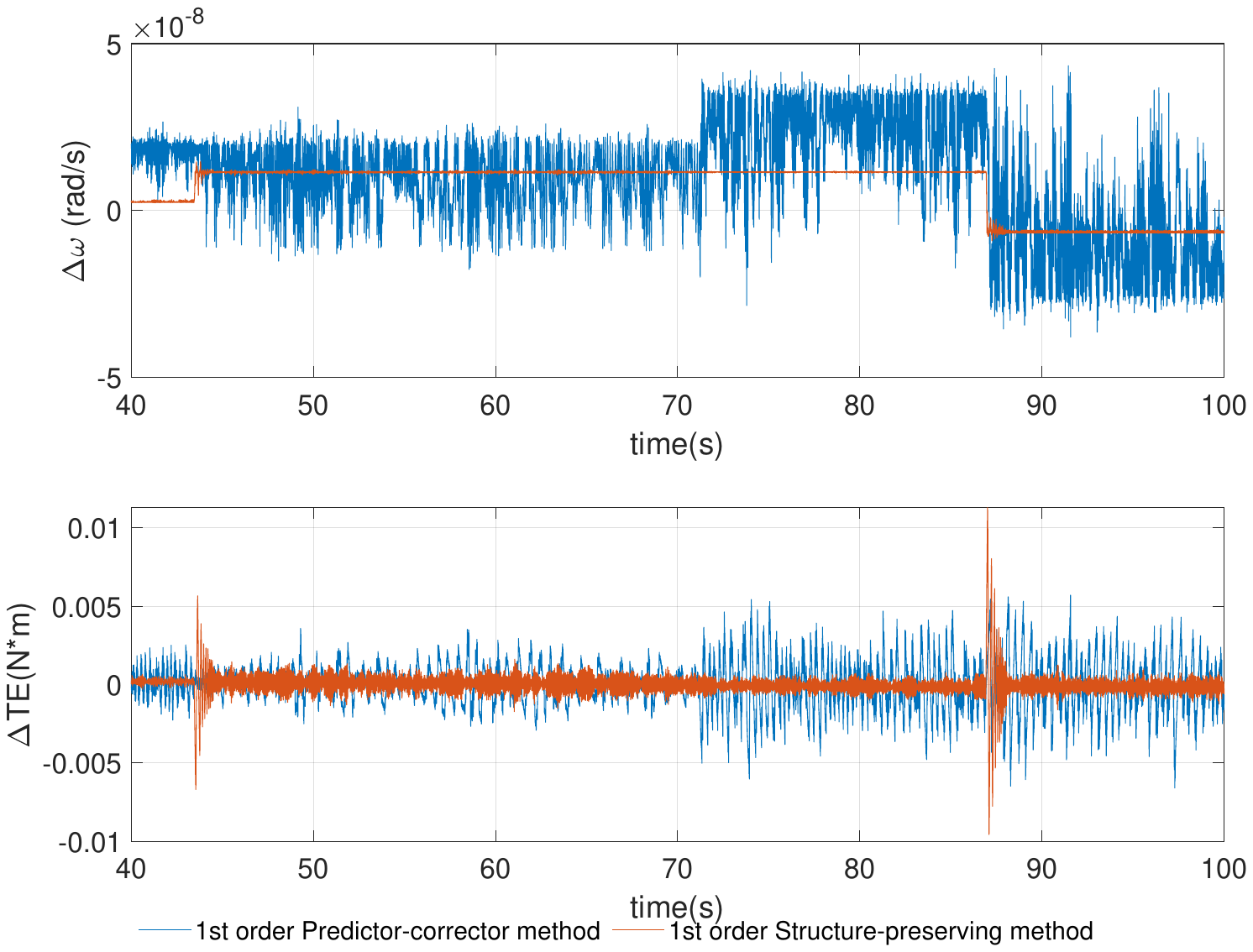}
 	\label{fig:PC0pt5}
 }

 \subfigure[2nd order methods in $0\sim 11\ \mathrm{s}.$]{
 	\includegraphics[scale=0.2]{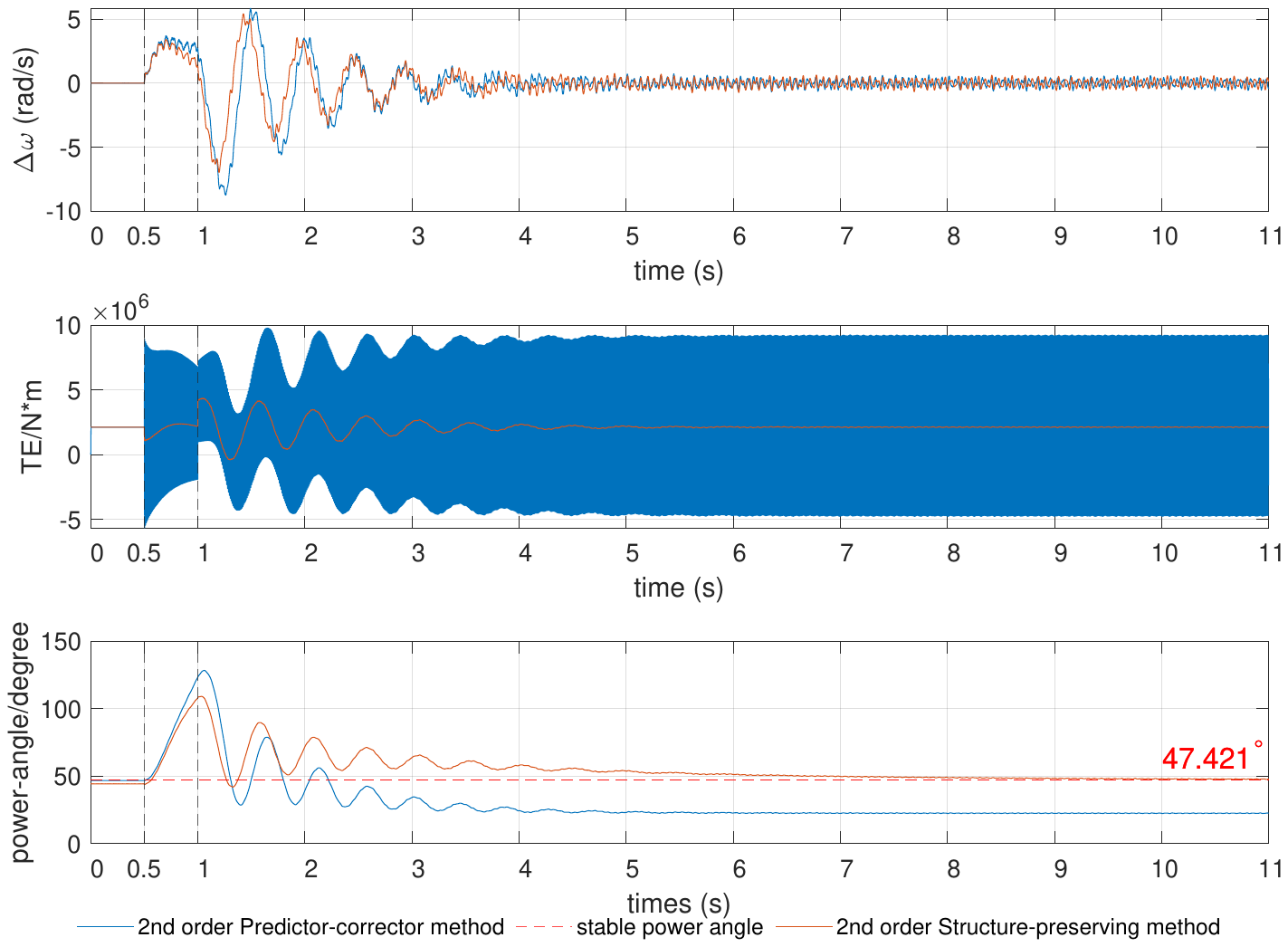}
 }
 \subfigure[2nd order methods in $500\sim 1000\ \mathrm{s}.$]{
 	\includegraphics[scale=0.2]{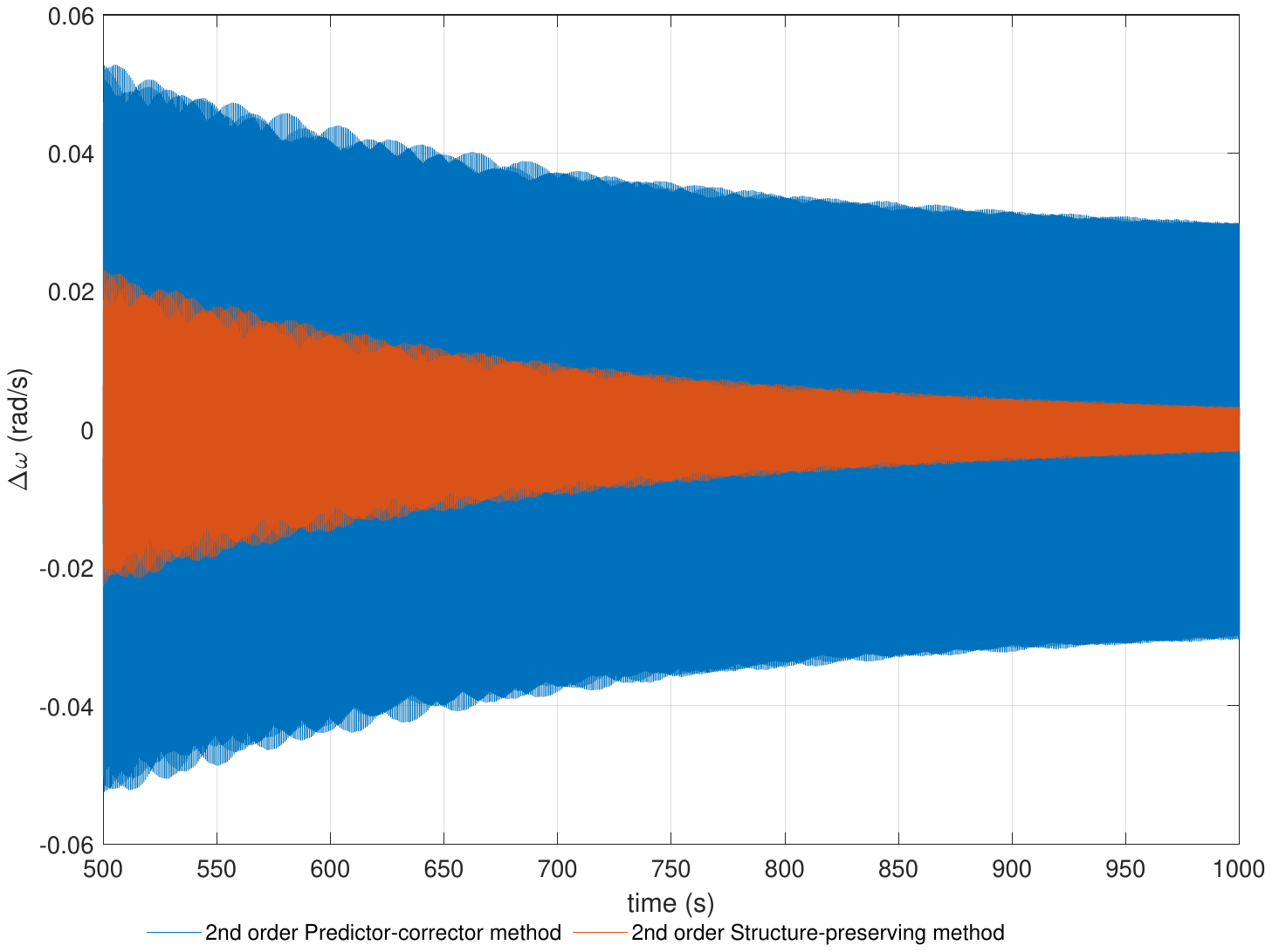}
 }
 \caption{Simulation of error of $5$-th angular velocity $\Delta\omega$, electromagnetic torque TE and power angle by two methods with break time $t_b=0.5$ seconds.}
 \label{fig:tbreak5}
\end{figure}

Here $\Delta\om:=\dot\theta_5-120\pi$, $T_{\mathrm{E}}:=\frac{1}{2}\bm{\Psi}^\top \frac{dN}{d\theta}\bm{\Psi}=\frac{1}{2}\widetilde{\bm{\Psi}}^\top \frac{dN}{d\theta}\widetilde{\bm{\Psi}}$ is called the electromagnetic torque and the power angle is defined as $\theta_5-\angle(\dot\Psi_{1\alpha}+\mathrm{i}\cdot \dot\Psi_{1\beta})$, i.e. the difference between the angle $\theta_5$ of the fifth mass block in generator and the argument of the voltage $\dot\Psi_{1\alpha}+\mathrm{i}\cdot \dot\Psi_{1\beta}\in \mathbb{C}$ at node $1$. 

Recall that the angle degree $47.421^\circ$ represents the equilibrium point of stage III, i.e. the value to which the power angle should converge if the circuit system finally tends to stability. This indicates that the structure preserving method performs better in long time compared to the P-C method, since the power angle tends to $47.421^\circ$ in figure \ref{fig:Gauss0pt5} but not in figure \ref{fig:PC0pt5}. 

Also, a comparison of these two methods for power-angle in $1500$ seconds is shown in figure \ref{fig:longang}, which shows that our structure preserving method behaves more stable long timely.

\begin{figure}[ht]
	\centering
	\subfigure[2nd order Structure Preserving method for power angle in $1500\ \mathrm{s}$.]{
		\includegraphics[scale=0.45]{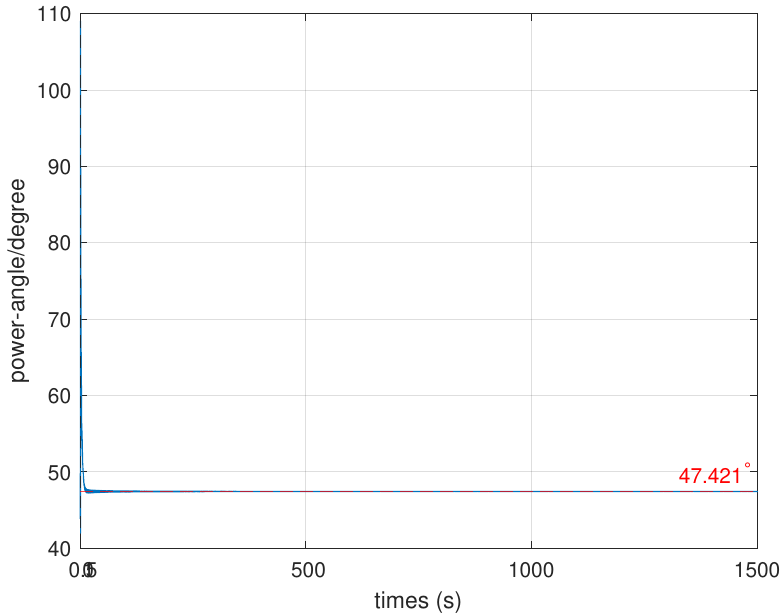}
	}
	\subfigure[2nd order P-C method for power angle in $1500\ \mathrm{s}$.]{
		\includegraphics[scale=0.45]{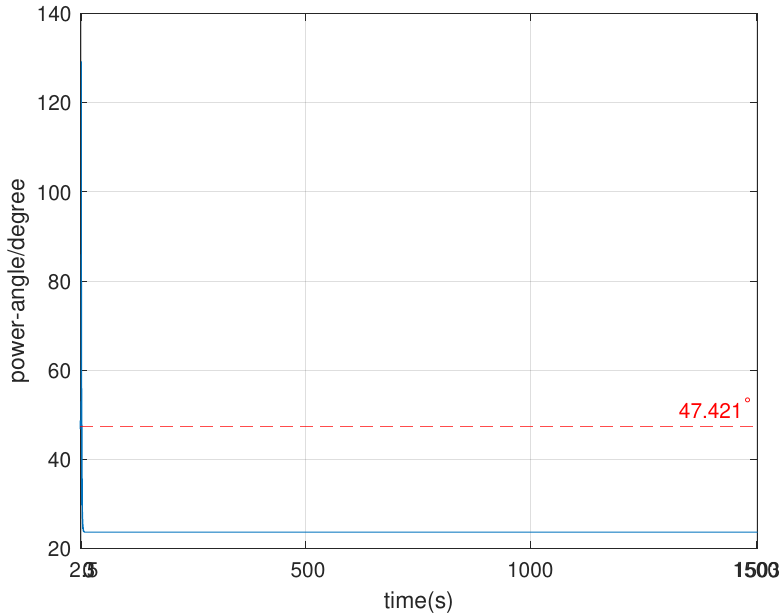}
	}
	\caption{Simulation of power angle in $1500$ seconds.}
	\label{fig:longang}
\end{figure}

Thus, to simulate the critical clearing time (CCT), we apply the structure preserving method only. 
\begin{figure}[ht]
	\centering
	\subfigure[2nd order Structure Preserving method for $t_b=0.77\ \mathrm{s}$.]{
		\includegraphics[scale=0.45]{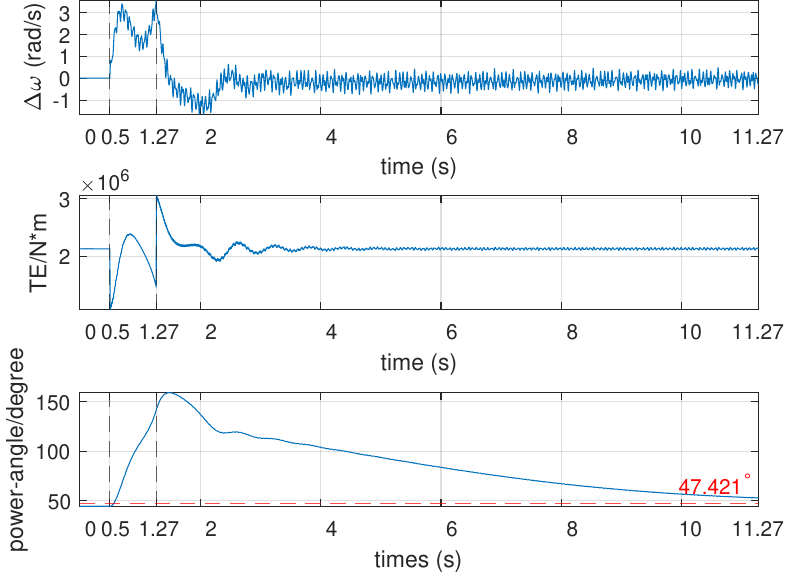}
	}
	\subfigure[2nd order Structure Preserving method for $t_b=0.78\ \mathrm{s}$.]{
		\includegraphics[scale=0.45]{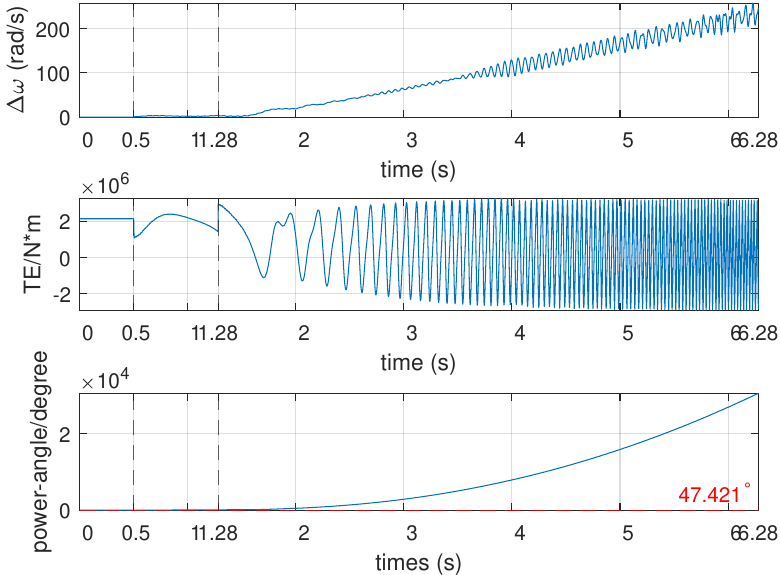}
	}
	\caption{Simulation for CCT by Structure Preserving method.}
	\label{fig:CCT}
\end{figure}
In this way, we can conclude from figure \ref{fig:CCT} that the critical clearing time for the circuit system \label{specialcase} is about $0.77$ seconds.
\section{Conclusions}\label{sec:conclusion}
In this article we newly establish a reductive form of a fault transient circuit system, which contains three stages and two switching processes, and apply a structure-preserving method on it. We accelerate the reduction process of this structure-preserving method by showing that a key matrix $\wt{N}(\theta)$ is a linear combination of $\sin\theta,\dots,\cos2\theta$. Also we rigorously derive a switching method to decide the initial condition of a stage from the final condition of the former stage, so that the circuit values maintain stable. A simplified circuit system is applied with our structure-preserving method and a predictor-corrector method simultaneously and the numerical results show that the corresponding physical quantities including angular velocities, electromagnetic torque and power angle are simulated more precisely by our method. In this sense, by using the structure-preserving method, we conclude that the critical clearing time for the system is about $0.77$ seconds. In addition, the system \eqref{eq:oldsys} possesses a range of geometric structures, including the Birkhoffian structure, which is employed to characterize dissipative systems. The development of Birkhoffian numerical methods will be addressed in future research.
\section*{Acknowledgments}
This study is supported by State Key Laboratory of Advanced Power Transmission Technology (Grant No. GEIRI-SKL-2023-006), State Key Laboratory of Alternate Electrical Power System with Renewable Energy Sources (Grant No. LAPS23003) and National Natural Science Foundation of China (Grant No. 12171466).

\bibliographystyle{plain}
\bibliography{ref}

\begin{thebibliography}{10}

\bibitem{EPPC}
E~Celledoni and E~H Høiseth.
\newblock Energy-preserving and passivity-consistent numerical discretization
  of port-{H}amiltonian systems.
\newblock {\em arXiv preprint}, arXiv:1706.08621, 2017.

\bibitem{TMS}
S~Cheng, Y~Cao, and Q~Jiang.
\newblock {\em Theory and Method of Subsynchronous Oscillation in Power
  System}.
\newblock Science Press, Beijing, 2009.

\bibitem{DCS}
H~W Dommel.
\newblock Digital computer solution of electromagnetic transients in single and
  multiphase networks.
\newblock {\em IEEE Trans Power Appar Syst}, 88:388--399, 1969.

\bibitem{EMTP}
H~W Dommel.
\newblock {\em EMTP theory book}.
\newblock Microtran Power System Analysis Corporation, Vancouver, British
  Columbia, 2nd edition, 1992.

\bibitem{RHED}
Y~Dong, Y~Wang, J~Han, Y~Li, S~Miao, and J~Hou.
\newblock Review of high efficiency digital electromagnetic transient
  simulation technology in power system.
\newblock {\em Proc CSEE}, 38:2213--2231, 2018.

\bibitem{HASM}
B~L Ehle.
\newblock High order {A}-stable methods for the numerical solution of systems
  of {DE}s.
\newblock {\em Bit Numer Math}, 8:276--278, 1968.

\bibitem{S1}
K.~Feng.
\newblock On difference schemes and symplectic geometry.
\newblock In K.~Feng, editor, {\em Proceedings of 1984 Beijing Symposium on
  Differential Geometry and Differential Equations}, pages 42--58, Beijing,
  1985. Science Press.

\bibitem{NSDAS}
E~Hairer, C~Lubich, and M~Roche.
\newblock {\em The Numerical Solution of Differential-Algebraic Systems by
  Runge-Kutta Methods}.
\newblock Springer-Verlag, Berlin, 1989.

\bibitem{GNI}
E~Hairer, C~Lubich, and G~Wanner.
\newblock {\em Geometric Numerical Integration: Structure Preserving Algorithms
  for Ordinary Differential Equations}, pages 179--195.
\newblock Springer-Verlag, Berlin, 2nd edition, 2006.

\bibitem{SODE1}
E~Hairer, S~P Nørsett, and G~Wanner.
\newblock {\em Solving Ordinary Differential Equations I: Nonstiff Problems},
  pages 356--360.
\newblock Springer-Verlag, Berlin, 2nd edition, 1993.

\bibitem{SODE2}
E~Hairer and G~Wanner.
\newblock {\em Solving Ordinary Differential Equations II: Stiff and
  Differential-Algebraic Problems}.
\newblock Springer-Verlag, Berlin, 2nd edition, 1996.

\bibitem{CPD1}
Y~He, Z~Zhou, Y~Sun, J~Liu, and H~Qin.
\newblock Explicit {K}-symplectic algorithms for charged particle dynamics.
\newblock {\em Phys Lett A}, 381:568--573, 2017.

\bibitem{MA}
Roger~A. Horn and Charles~R. Johnson.
\newblock {\em Matrix Analysis}.
\newblock Cambridge, 2nd edition, 2013.

\bibitem{FBMSR}
{IEEE Committee}.
\newblock First benchmark model for computer simulation of subsynchronous
  resonance.
\newblock {\em IEEE Trans Power Appar Syst}, 96:1565--1572, 1977.

\bibitem{DTP}
F~Ji, L~Gao, and C~Lin.
\newblock Dynamics of three phase {AC} system and {VSC} access problem
  research.
\newblock {\em Proc CSEE}, 42:2286--2298, 2022.

\bibitem{LMMS}
F~Ji, L~Gao, C~Lin, and Y~Liu.
\newblock Lagrangian modelling and motion stability of synchronous generator
  power systems.
\newblock {\em arXiv preprint}, arXiv:2311.03737, 2023.

\bibitem{NDE}
F~Ji, Y~Qiu, X~Wei, X~Wu, and Z~He.
\newblock Nodal dynamic equation used for electromagnetic transient simulation
  of linear switching circuit.
\newblock {\em IET Sci, Meas Technol}, 12:626--633, 2018.

\bibitem{PSSC}
P~Kundur.
\newblock {\em Power System Stability and Control}.
\newblock McGraw-hill, New York, 1994.

\bibitem{DAE}
P~Kunkel and V~Mehrmann.
\newblock {\em Differential-Algebraic Equations. Analysis and Numerical
  Solution}.
\newblock European Mathematical Society Publishing House, Z\"urich, 2006.

\bibitem{pHDAE}
V~Mehrmann and R~Morandin.
\newblock Structure-preserving discretization for port-{H}amiltonian descriptor
  systems.
\newblock In {\em 2019 IEEE 58th Conference on Decision and Control (CDC)},
  pages 6863--6868, Nice, France, 2019.

\bibitem{S2}
J~M Sanz-Serna.
\newblock Symplectic integrators for {H}amiltonian problems: an overview.
\newblock {\em Acta Numer}, 1:243--286, 1992.

\bibitem{S3}
Y~Tang, V~M Pérez-García, and L~Vázquez.
\newblock Symplectic methods for the {A}blowitz-{L}adik model.
\newblock {\em Appl Math Comput}, 2:17--38, 1997.

\bibitem{CPD2}
M~Tao.
\newblock Explicit high-order symplectic integrators for charged particles in
  general electromagnetic fields.
\newblock {\em J Comput Phys}, 327:245--251, 2016.

\bibitem{PHST}
A~J {Van Der Schaft} and D~Jeltsema.
\newblock Port-{H}amiltonian systems theory: An introductory overview.
\newblock {\em Found Trends Syst Control}, 1:173--378, 2014.

\bibitem{PSETS}
N~Watson and J~Arrillaga.
\newblock {\em Power systems electromagnetic transients simulation}.
\newblock The Institution of Engineering and Technology, London, 2nd edition,
  2019.

\bibitem{PMRA}
Z~Xu.
\newblock Physical mechanism and research approach of generalized synchronous
  stability for power systems.
\newblock {\em Electr Power Autom Equip}, 40:3--9, 2020.

\bibitem{ENSS}
J~Zhang, A~Zhu, J~Feng, L~Chang, and Y~Tang.
\newblock Effective numerical simulations of synchronous generator system.
\newblock {\em Simulation: Transactions of the Society for Modeling and
  Simulation International}, 100:595--611, 2023.

\bibitem{NLS2}
R~Zhang, J~Huang, Y~Tang, and L~Vázquez.
\newblock Revertible and symplectic methods for the {A}blowitz-{L}adik discrete
  nonlinear {S}chrodinger equation.
\newblock In R~Crosbie, H~Vakilzadian, T~Ericsen, and Others, editors, {\em
  Proceedings of the 2011 Grand Challenges on Modeling and Simulation
  Conference}, pages 297--306, Hague, Netherlands, 2011.

\bibitem{GCD2}
R~Zhang, J~Liu, Y~Tang, H~Qin, and B~Zhu.
\newblock Canonicalization and symplectic simulation of the gyrocenter dynamics
  in time-independent magnetic fields.
\newblock {\em Phys Plasmas}, 21:032504, 2014.

\bibitem{CPD3}
R~Zhang, Y~Wang, Y~He, J~Xiao, J~Liu, H~Qin, and Y~Tang.
\newblock Explicit symplectic algorithms based on generating function for
  relativistic charged particle dynamics in time-dependent electromagnetic
  field.
\newblock {\em Phys Plasmas}, 25:022117, 2018.

\bibitem{DAA}
Y~Zhang, A~M Gole, W~Wu, B~Zhang, and H~Sun.
\newblock Development and analysis of applicability of a hybrid transient
  simulation platform combining {TSA} and {EMT} elements.
\newblock {\em IEEE Trans Power Syst}, 28:357--366, 2013.

\bibitem{CPD4}
Z~Zhou, Y~He, Y~Sun, J~Liu, and H~Qin.
\newblock Explicit symplectic methods for solving charged particle
  trajectories.
\newblock {\em Phys Plasmas}, 24:052507, 2017.

\bibitem{GCD1}
B~Zhu, Z~Hu, Y~Tang, and R~Zhang.
\newblock Symmetric and symplectic methods for gyrocenter dynamics in
  time-independent magnetic fields.
\newblock {\em Int J Model Simul Sci Comput}, 7:1650008, 2016.

\bibitem{EKS}
B~Zhu, L~Ji, A~Zhu, and Y~Tang.
\newblock Explicit {K}-symplectic methods for nonseparable non-canonical
  {H}amiltonian systems.
\newblock {\em Chin Phys B}, 32:020204, 2023.

\bibitem{NLS1}
B~Zhu, Y~Tang, R~Zhang, and Y~Zhang.
\newblock Symplectic simulation of dark solitons motion for nonlinear
  {S}chrodinger equation.
\newblock {\em Numer Algorithms}, 81:1485--1503, 2019.

\end{thebibliography}
\nocite{*}

\end{document}